\documentclass[12pt,oneside,reqno]{amsart}

\usepackage{amsthm, amssymb, amsmath, amsfonts}
\usepackage{enumerate}
\usepackage[numbers,sort&compress]{natbib}

\usepackage[colorlinks]{hyperref}

\newtheorem{thm}{Theorem}[section]

\newtheorem{lem}[thm]{Lemma}
\newtheorem{prop}[thm]{Proposition}
\newtheorem{corl}[thm]{Corollary}
\newtheorem{rem}[thm]{Remark}

\newtheorem{exm}{Example}
\allowdisplaybreaks
\title[Induced Representation of groupoids]{Induced Representation of Topological groupoids}
\author[K. N. Sridharan]{K. N. Sridharan}
\address{K. N. Sridharan,\newline\indent Department of Mathematics,\newline\indent Indian Institute of Technology Delhi,\newline\indent New Delhi - 110016, India.}
\email{sreedu242@gmail.com}
\author[N. S. Kumar]{N. Shravan Kumar}
\address{N. Shravan Kumar,\newline\indent Department of Mathematics,\newline\indent Indian Institute of Technology Delhi,\newline\indent New Delhi - 110016, India.}
\email{shravankumar.nageswaran@gmail.com}

\begin{document}
\begin{abstract}
Let $G$ be a locally compact second countable groupoid with a Haar system, and let $H$ be a closed wide subgroupoid of $G$ with a Haar system. Given a full equivariant system of measures $\mu=\{\mu^{u}\}_{u\in G^{0}}$ on $G/H$, we construct the representation of $G$ induced from a continuous unitary representation of $H$ on a continuous Hilbert bundle. No quasi-invariant measure on $G^0$ is needed. We show that such a system $\mu$ always exists when $G$ is \'etale, and that the induced representation does not change, up to unitary equivalence, when $\mu$ is replaced by $c\,\mu$ for a suitable positive function $c$. For an invariant probability measure on $G^0$, we show that our induced representation agrees with the one obtained by Rieffel's method. In particular, Rieffel's method preserves continuity in this setting. Finally, we prove basic properties of induction, induction in stages, a groupoid version of Mackey's tensor product theorem, and a version of the Frobenius reciprocity theorem for compact transitive groupoids.
\end{abstract}
\keywords{Groupoids, Induced representation, Frobenius Reciprocity,Continuous Hilbert bundles, Hilbert modules}
\subjclass{Primary 18B40, 22A22, 22D30; Secondary 46L08}
\maketitle
\section{Introduction}
The concept of induced representation plays an important role in
representation theory. Frobenius introduced it for finite groups, and Mackey
extended it to locally compact groups in a series of papers
\cite{mackey1949imprimitivity,mackey1951induced,mackey1952induced,mackey1953induced}.
Mackey's construction \cite{mackey1952induced} uses a quasi-invariant measure
on the homogeneous space $G/H$, and the induced representations obtained from
two such measures are unitarily equivalent
\cite[Theorem 2.1]{mackey1952induced}. Blattner \cite{blattner1961induced}
extended the theory to the non-separable case, where such a measure need not
exist. We refer to \cite{kaniuth2013induced} for more details.

Induced representations have also been studied for hypergroups
\cite{hermann1992induced} and for groupoids. In \cite{hermann1992induced}, the induced
representation is constructed from a class of representations of closed
subhypergroups called inducible representations. For measured groupoids,
representations act on measurable Hilbert bundles over $G^0$, and Renault
\cite{renault2006groupoid} introduced induced representations from closed
subgroupoids, following Rieffel's approach \cite{rieffel1974induced}. Ionescu
and Williams \cite{MR2465655} formalized the theory of inducing from a general
closed subgroupoid, using the full form of the disintegration theorem
\cite[Theorem 7.8]{MR2547343}, and showed that a representation induced from
an irreducible representation of a stability group is irreducible. In these
works, a quasi-invariant measure on $G^0$ is used, and the induced
representations are representations of $C^*(G)$. Pysiak \cite{Pysiak2004}
studied representations of finite groupoids through systems of
imprimitivity, and in \cite{Pysiak2011} he studied representations of locally
compact, transitive and locally trivial groupoids induced from the isotropy
subgroupoid, also in the measurable setting, and proved an imprimitivity theorem.
For groupoids without topology, induction and Frobenius reciprocity were
studied by Barbar\'an S\'anchez and El Kaoutit \cite{BarbaranElKaoutit}.

In this paper, we induce continuous unitary representations of a closed wide
subgroupoid $H$ of a topological groupoid $G$ on continuous Hilbert bundles,
in the setting of Bos \cite{bos2011continuous}. Our groupoids need not be
transitive or locally trivial. In place of a quasi-invariant measure, we use
a full equivariant system of measures $\mu$ on $G/H$, which plays the role of
the invariant measure on $G/H$ in the group case. No measure on $G^0$ is
needed. Such a system exists, for instance, when $G/H$ is a proper $G$-space,
or when $G$ is \'etale (Proposition 3.10). We use the correspondence between
Hilbert $C_0(G^0)$-modules and continuous fields of Hilbert spaces
\cite{BSMF_1963__91__227_0,bos2007groupoids} and continuity results from
\cite{bos2011continuous}. Several steps need new arguments in the continuous
setting (Proposition 3.2, Lemma 3.4, Proposition 3.5 and Lemma 3.6). Unlike
the group case, the induced representation can depend on $\mu$ (Example 3).
It does not, up to unitary equivalence, when two systems differ by a suitable
positive function, in particular for \'etale groupoids (Proposition 3.10 and
Remark 3.11). We also show that the left regular representation is induced
from the unit space (Example 1), and we compute a representation of the
rotation groupoid $\mathbb{Z}\ltimes\mathbb{T}$ induced from
$2\mathbb{Z}\ltimes\mathbb{T}$, whose operators change from point to point of
the unit space (Example 2).

Section 2 recalls basic results on groupoids and their representations.
Section 3 contains the construction, its basic properties, the \'etale case
and the examples. In Section 4, we prove induction in stages for closed wide
subgroupoids $K\subseteq H\subseteq G$, and in Section 5, Mackey's tensor
product theorem for groupoids, using tensor products of Hilbert modules.
Frobenius reciprocity has several versions for groups
\cite{mackey1953induced,mautner1951generalization,Moore1962OnTF} and a version
for compact hypergroups \cite[Theorem 4]{hermann1992induced}. In Section 6, we
prove a version for compact transitive groupoids: the spaces
$Mor\big(\pi,(\operatorname{ind}_{H}^{G}(\sigma),\mu)\big)$ and
$Mor(\pi|_{H},\sigma)$ are isomorphic. The idea comes from compact groups, but
the proof requires identifying the fibres of the induced bundle with induced
spaces of the isotropy groups, and showing that the intertwining operators
are continuous. In Section 7, for an invariant probability measure on $G^0$,
we show that the representation of $C^*(G)$ induced by Rieffel's method, as in
\cite{MR2465655}, is unitarily equivalent to the integrated form of
$\operatorname{ind}_H^G(\sigma,\mu)$ (Theorem 7.1). This is a groupoid
analogue of \cite[Theorem 5.12]{rieffel1974induced}, and it shows that
Rieffel's method preserves continuity in this setting.
\section{Preliminaries}
Let's start with some basics on groupoids and their representations.

A groupoid is a set $G$ endowed with a product map $G^{2}\to G: (x,y)\to xy$, where  $G^{2}$ is a subset of $G\times G$ called the set of composable pairs, and an inverse map  $G\to G:x\to x^{-1}$ such that the following relations are satisfied: 
    \begin{enumerate}[(i)]
        \item $(x^{-1})^{-1}=x$,
        \item $(x,y),(y,z)\in G^{2}$ implies $(xy,z),(x,yz)\in G^{2}$ and $(xy)z=x(yz)$,
        \item $(x^{-1},x)\in G^{2}$ and if $(x,y)\in G^{2}$, then $x^{-1}(xy)=y$,
        \item $(x,x^{-1})\in G^{2}$ and if $(z,x)\in G^{2}$, then $(zx)x^{-1}=z.$
    \end{enumerate}
  If $x \in G,~ d,r: G\to G$, defined as $d(x) = x^{-1}x$ and $r(x)= xx^{-1}$ are its  domain and range maps respectively. The image of range and domain maps denoted $G^{0}$ is called the unit space of $G$. For $u \in G^{0}, G^{u} = r^{-1}\{u\}$ and $G_{u} = d^{-1}\{u\}$. The set $G^{u}_{v}=G^{u}\cap G_{v}$ and for every $u\in G^{0}$, $G^{u}_{u}$ has a group structure with identity $u$, known as isotropy subgroup at $u$.
  
    A topological groupoid consists of a groupoid $G$ and a topology compatible with the groupoid structure such that:
    \begin{enumerate}[(i)]
        \item $x\to x^{-1}:G\to G$ is continuous,
        \item$(x,y)\to xy: G^{2}\to G$ is continuous where $G^{2}$ has the induced topology from $G \times G$. 
    \end{enumerate}
     Here we consider second countable locally compact Hausdorff groupoids. The unit space $G^{0}$ is a locally compact Hausdorff space under the subspace topology. Both range and domain maps are continuous.
     
     The notion of Haar system of topological groupoids is analogous to the Haar measure in groups and we assume the existence of Haar system. The following is the definition of Haar system.
     
     A left Haar system for $G$ consists of a family  $\{\lambda^{u}: u\in G^{0}\}$ of positive Radon measures on $G$ such that, 
\begin{enumerate}[(i)]
    \item the support of the measure $\lambda^{u}$ is $G^{u}$,
    \item for any $f\in C_{c}(G),u\to \lambda(f)(u):= \int f d\lambda^{u}$ is continuous and
    \item for any $x\in G$ and $f \in C_{c}(G)$,\[\int_{G^{d(x)}}f(xy)d\lambda^{d(x)}(y)=\int_{G^{r(x)}}f(y)d\lambda^{r(x)}(y).\]  
\end{enumerate}
According to \cite[Proposition $2.4$]{renault2006groupoid}, if $G$ is a locally compact groupoid with a left Haar system, then the range map $r$ is an open map. More details on groupoid can be referred to\cite{renault2006groupoid,paterson2012groupoids}

Now we define the left action of a groupoid.

   Suppose  $G$ is a topological groupoid and  $X$ is a locally compact Hausdorff space together with a continuous map $r_{X}: X \to G^{0}$ called the moment map. Then a left action of $G$ on $X$ is a continuous map $(\gamma,x) \to \gamma\cdot x$ from $G*X= \{(\gamma,x) \in G\times X : s(\gamma) = r_{X}(x) \}$ to X such that
   \begin{enumerate}[(i)]
       \item $r_{X}(x) \cdot x= x~~ \forall x \in X$ and
       \item if $(\gamma,\eta) \in  G^{2}$ and $(\eta,x) \in  G*X$, then $(\gamma,\eta \cdot x) \in  G*X$ and $\gamma \eta \cdot x=\gamma \cdot (\eta \cdot x)$.
   \end{enumerate} 
   $X$ is called left $G$-space.

If $Y$ and $X$ are both left $G$-spaces and $\pi$ is equivariant, i.e. $\pi(g\cdot y)= g\cdot \pi(y)$, then  a fully equivariant $\pi$-
system $\beta$ is a family of measures $\{\beta^{x}: x\in X\}$ on $Y$ such that
\begin{enumerate}[(i)]
    \item $supp(\beta^{x})= \pi^{-1}(x)$
    \item $x\to \beta(f)(x)= \int_{Y} f(y)d\beta^{x}(y)$ is continuous for $f\in C_{c}(Y)$
    \item For $f\in C_{c}(Y) , (\gamma,x)\in G*X$,
    $$ \int_{Y} f(\gamma \cdot y)d\beta^{x}(y)=\int_{Y} f(y)d\beta^{\gamma \cdot x}(y)$$
\end{enumerate} 
The left Haar system of a groupoid is an example of the fully equivariant $\pi$-system.

Now we define the continuous representation of groupoids over a continuous field of Hilbert spaces.
 A continuous field of Hilbert spaces over $G^{0}$ is a  family $\{\mathcal{H}_{u}\}_{u\in G^{0}}$ of Hilbert spaces, with a set $\Gamma\subset \prod_{u\in G^{0}}\mathcal{H}_{u}$ of vector fields such  that: 
\begin{enumerate}[(i)]
    \item $\Gamma$ is a complex linear subspace of $\prod_{u\in G^{0}}\mathcal{H}_u$.
    \item For every $u\in G^{0}$, the set $\xi(u)$ for $\xi \in \Gamma$ is dense in $\mathcal{H}_{u}$.
    \item For every $\xi \in \Gamma$, the function $u \to \|\xi(u)\|$ is continuous.
    \item Let $\xi \in \prod_{u\in G^{0}}\mathcal{H}_{u}$ be a vector field; if for every $u\in G^{0}$ and every $\epsilon >0$, there exists an $\xi'\in \Gamma$ such that $\|\xi(s) -\xi'(s)\| < \epsilon$ on a neighbourhood of $u$, then $\xi \in \Gamma$.
\end{enumerate}

       For a continuous field of Hilbert spaces, we can define a topology on $\mathcal{H}= \sqcup_{u\in G^{0}}H_{u}$, generated by the sets of the form 
\[U(V,\xi,\epsilon)=\{h\in \mathcal{E}:\|h-\xi(p(h))\|<\epsilon,\xi\in \Gamma,p(h)\in V\}.\]
    where $V$ is an open set in $G^{0}$, $\epsilon> 0$, and $p: \mathcal{H}\to G^{0}$ is the projection of the total  space $\mathcal{H}$ to base space $G^{0}$ such that fibre $p^{-1}(u) = \mathcal{H}_{u},u \in G^{0}$.  This map is a surjective continuous open map under the above topology. We denote such continuous field of Hilbert space as $(\mathcal{H},\Gamma)$.  With the above topology, it forms the structure of a continuous Hilbert bundle, and $\Gamma$ forms the continuous sections.  We denote the space of continuous sections vanishing at infinity as $C_{0}(G^{0},\mathcal{H})$. Since $G^{0}$ is locally compact, we can see that  $C_{0}(G^{0},\mathcal{H})$ is fibrewise dense and forms a Banach space in $\Gamma$ under supremum norm. According to \cite[Remark $13.19$]{fell1988representations}, a continuous Hilbert bundle over $G^{0}$ has enough continuous sections, i.e for every $b\in \mathcal{H}$, there exist a continuous section $\alpha$ such that $\alpha(p(b))=b$. A subset $\mathcal{K}$ of $\mathcal{H}$ is a subbundle if $p_{|_{\mathcal{K}}}$ is continuous onto open map and each fibre of $p^{-1}(u)\cap \mathcal{K}$ is a closed subspace of $\mathcal{H}$. The topology of $\mathcal{K}$ is inherited from $\mathcal{H}$. For more details, refer \cite{fell1988representations,gierz2006bundles,bos2011continuous}.
    
    A  continuous  representation of groupoid $G$ is a double $(\mathcal{H}^{\pi},\pi)$, where $\mathcal{H}^{\pi}=\{\mathcal{H}^{\pi}_{u}\}_{u\in G^{0}}$ is a continuous Hilbert bundle over $G^{0}$ such that  
\begin{enumerate}[(i)]
    \item  $\pi(x) \in \mathcal{B}\left(\mathcal{H}^{\pi}_{d(x)}, \mathcal{H}^{\pi}_{r(x)}\right)$ is a unitary operator, for each $x \in G$,

\item $\pi(u)$ is the identity map on $\mathcal{H}^{\pi}_{u}$ for all $u \in G^{0}$,

\item $\pi(x) \pi(y)=\pi(x y)$ for all $(x, y) \in G^{2}$,

\item $\pi(x)^{-1}=\pi\left(x^{-1}\right)$ for all $x \in G$,

\item $x \rightarrow \langle \pi(x) \xi(d(x)),\eta(r(x))\rangle$ is continuous for every $\eta,\xi \in C_{0}(G^{0},\mathcal{H}^{\pi})$.
\end{enumerate}
Suppose $(\mathcal{H}^{\pi},\pi)$ and $(\mathcal{H}^{\pi'},\pi')$ be two continuous representations of groupoid $G$, then $Mor(\pi,\pi')$ refers to the bundle morphisms $T:\mathcal{H}^{\pi} \to \mathcal{H}^{\pi'}$ such that 
\[ T_{r(x)}\pi(x)=\pi'(x)T_{d(x)},~~ x\in G.\]
Note that $Mor(\pi,\pi')$ is a vector space.

We say that two representations $\pi$ and $\pi'$ are unitarily equivalent, denoted $\pi \sim \pi'$, if there exist $T \in Mor(\pi,\pi')$ such that each $T_{u}$ is unitary operator. A representation $(\mathcal{H}^{\pi},\pi)$ is called reducible if there is a subbundle $\mathcal{K}$ invariant under $\pi$, i.e for every $u$, $\mathcal{K}_{u}$ is a closed subspace of $\mathcal{H}^{\pi}_{u}$ and $\pi(x)\mathcal{K}_{d(x)} \subseteq \mathcal{K}_{r(x)}$. It is irreducible if it is not reducible.

\section{Construction of Induced representation}
Let $H$ be a wide closed subgroupoid of a second countable locally compact groupoid $G$ with Haar system $\{\lambda_{H}^{u}\}_{u\in G^{0}}$. We can easily verify that $G/H$ is a left $G$-space.  By the arguments similar to \cite[Proposition $2.1$]{renault2006groupoid}, $G/H$ is a locally compact Hausdorff space under quotient topology with quotient map $q_{H}$ and the moment map $r_{G^{0}}: G/H \to G^{0}, r_{G^{0}}(gH)=r(g)$ is an open continuous map. We assume and fix  a full equivariant $r_{G^{0}}$- system, $\mu=\{\mu^{u}\}_{u\in G^{0}}$ on the left $G$- space $G/H$.

For $f\in C_{c}(G)$, we can define a function $Pf$ in $C_{c}(G/H)$, such that
$$ Pf(xH)=\int_{H}f(x\xi)d\lambda_{H}^{d(x)}(\xi).$$
It is well defined, due to the invariance of Haar system $\{\lambda^{u}_{H}\}_{u\in G^{0}}$ and $supp(Pf) \subset q(supp(f))$. The continuity of the function can be observed using \cite[Lemma $3.12$]{bos2011continuous}.
Now, we prove an important lemma which is later being used in this paper.
\begin{lem}
    If $J \subset G/H$ is compact, there exists $f\geq 0$ in $C_{c}(G)$ such that $Pf=1$ on $J$.
\end{lem}
\begin{proof}
    Let $E$ be a compact neighbourhood of $J$ in $G/H$. There exist a compact $K\subset G$ such that $q_{H}(K)=E$. Choose a non negative $g\in C_{c}(G)$ with $g>0$ on $K$ and $\phi \in C_{c}(G/H)$ supported in $E$ such that $\phi =1$ on $J$. Define a continuous function $$ f= \frac{\phi \circ q_{H}}{Pg\circ q_{H}}g.$$ Here $f$ is continuous since $Pg> 0$ on $supp(\phi)$, its support is contained in $supp (g)$ and $Pf = \phi$. 
\end{proof}
 Suppose $(\mathcal{H}^{\sigma},\sigma)$ is a continuous unitary representation of $H$.  Let $C(G,\mathcal {H}^{\sigma})$ denotes the set of  continuous function $f$ such that $f(x) \in \mathcal{H}^{\sigma}_{d(x)}$ and its subspace $C_{c}(G,\mathcal {H}^{\sigma})$ be the functions with compact support. For $f\in C(G,\mathcal {H}^{\sigma}) $, we also denote $f_{|_{G^{u}}}$ as $f^{u}$.
Define,
\begin{align*}
  \mathcal{F}_{0}^{\sigma}(G)=\Bigg\{f\in C(G,\mathcal {H}^{\sigma}):~ & q_{H}(supp(f))~ \text{is compact}~ \text{and} \\&f(x\xi)=\sigma(\xi^{-1})f(x),\text{for}~ (x,\xi)\in G^{2},\xi \in H \Bigg \}. 
\end{align*}

The next proposition provides explicitly the structure of functions in $\mathcal{F}_{0}^{\sigma}(G)$.
\begin{prop}
    If $\alpha \in C_{c}(G,\mathcal {H}^{\sigma})$, then the function
    $$f_{\alpha}(x)= \int_{H} \sigma(\eta)\alpha(x\eta)d\lambda_{H}^{d(x)}(\eta)$$
    belongs to $\mathcal{F}_{0}^{\sigma}(G)$. Moreover, every element of $\mathcal{F}_{0}^{\sigma}(G)$ is of the form $f_{\alpha}$ for some $\alpha \in C_{c}(G,\mathcal{H}^{\sigma})$.
\end{prop}
 \begin{proof}
 For each $x\in G$, the above integral makes sense due to \cite[Lemma $3.4$,$3.5$]{bos2011continuous} and \cite[Appendix Theorem $A.20$]{folland2016course}.
     Also it is clear that $q_{H}(supp~ (f_{\alpha}))\subset q_{H}(supp~ (\alpha))$, and for $(x,\xi) \in G^{2}, \xi \in H$, we have 
     \[ f_{\alpha}(x\xi)= \int_{H}\sigma(\eta)\alpha(x\xi\eta)d\lambda_{H}^{d(\xi)}(\eta)= \int_{H} \sigma(\xi^{-1}\eta)\alpha(x\eta)d\lambda^{r(\xi)}_{H}(\eta)= \sigma(\xi^{-1})f_{\alpha}(x).\]
     Next, we show that $f_{\alpha}$ is continuous. For that, it is enough to show that $x\to \|f_{\alpha}(x)\|_{\mathcal{H}^{\sigma}_{d(x)}}$ and $x \to \langle f_{\alpha}(x), t(d(x))\rangle_{H_{d(x)}}$ is continuous, where $t \in C_{0}(G^{0},\mathcal{H}^{\sigma})$.
     
     Since $\sigma$ is a unitary representation, by [Lemma $3.4$,$3.5$]\cite{bos2011continuous} $(x,\eta) \to \sigma(\eta)\alpha(x\eta)$ is continuous over $(G\times H)\cap G^{2}$. 
     Now, for $t \in C_{0}(G^{0},\mathcal{H}^{\sigma})$
     $$ \langle f_{\alpha}(x), t(d(x)\rangle_{\mathcal{H}^{\sigma}_{d(x)}}= \int_{H} \langle\sigma(\eta)\alpha(x\eta), t(d(x))\rangle_{\mathcal{H}^{\sigma}_{d(x)}} d\lambda_{H}^{d(x)}(\eta)$$
     Using \cite[Lemma $3.12$]{bos2011continuous}, we can say that $x\to \langle f_{\alpha}(x), t(d(x)\rangle_{H_{d(x)}}$ is continuous.
     Similarly,
     $$\|f_{\alpha}(x)\|^{2}_{\mathcal{H}^{\sigma}_{d(x)}}=\iint_{H}\langle \sigma(\eta)\alpha(x\eta),\sigma(k)\alpha(xk)\rangle_{\mathcal{H}^{\sigma}_{d(x)}}d\lambda_{H}^{d(x)}(k)d\lambda_{H}^{d(x)}(\eta)$$ is continuous .

 Suppose $f \in \mathcal{F}_{0}^{\sigma}(G)$, then by Lemma $3.1$, there exist $\Psi \in C_{c}(G)$, such that $\int_{H}\Psi(x\eta)d\lambda_{H}^{d(x)}(\eta)= 1$ for $x\in supp(f)$. Let $\alpha= \Psi f$, then
 $$ f_{\alpha}(x)=\int_{H}\Psi(x\eta)\sigma(\eta)f(x\eta)d\lambda_{H}^{d(x)}= \int_{H}\Psi(x\eta)f(x)d\lambda_{H}^{d(x)}(\eta)=f(x).$$
 So, $f=f_{\alpha}.$
 \end{proof}

 Since $\sigma$ is a unitary representation, we can easily see that for $f,g\in \mathcal{F}_{0}^{\sigma}(G)$, the function $x\to \langle f(x),g(x)\rangle_{\mathcal{H}^{\sigma}_{d(x)}}$ defines a function in $C_{c}(G/H)$ and thus we can define a function on $C_{0}(G^{0})$, 
 as
 $$ \langle f,g\rangle(u)=\int_{G/H}\langle f(x),g(x)\rangle_{\mathcal{H}^{\sigma}_{d(x)}}d\mu^{u}(xH).$$
 Define a norm on $\mathcal{F}_{0}^{\sigma}(G)$, as 
 
 For $f\in \mathcal{F}_{0}^{\sigma}(G)$
 $$ \|f\|= \sup_{u\in G^{0}}\sqrt{\langle f,f\rangle(u)}.$$ 
 Let $\mathcal{F}^{\sigma}(G)$ be the completion of $\mathcal{F}_{0}^{\sigma}(G)$ under the above norm. It is also denoted as $\mathcal{F}^{\sigma}(G,\mu)$, when $\mu$ need to be specified.
 
 In the next lemma we show that $\mathcal{F}^{\sigma}(G)$ form a left Hilbert $C_{0}(G^{0})$-module and if each fibre of $\mathcal{H^{\sigma}}$ is nonzero, it forms a full $C_{0}(G^{0})$-module.
 \begin{lem}\label{lem3.1.3}
     The Banach space $\mathcal{F}^{\sigma}(G)$ form  a  left Hilbert $C_{0}(G^{0})$- module under the following action:$$ \text{For}~~b\in C_{0}(G^{0}), f\in \mathcal{F}^{\sigma}(G), ~~bf(x)=b(r(x))f(x). $$ It forms a full left Hilbert $C_{0}(G^{0})$- module  if each fibre of $\mathcal{H}^{\sigma}$ is nonzero.
 \end{lem}
 \begin{proof}
     Let $f_{\alpha}\in \mathcal{F}_{0}^{\sigma}(G)$ and $b\in C_{0}(G^0)$, then by definition
     \begin{align*}
         bf_{\alpha}(x)= b(r(x))f_{\alpha}(x)&= b(r(x))\int_{H}\sigma(\xi)\alpha(x\xi)d\lambda_{H}^{d(x)}(\xi) \\
         &= \int_{H} \sigma(\xi)b(r(x\xi))\alpha(x\xi)d\lambda_{H}^{d(x)}(\xi)\\
         &=\int_{H}\sigma(\xi)(b\alpha)(x\xi)d\lambda_{H}^{d(x)}(\xi)= f_{b\alpha}(x).   
     \end{align*}
     Thus, $$\langle bf_{\alpha},bf_{\alpha}\rangle(u)= \int_{G/H} \langle bf_{\alpha}(x),bf_{\alpha}(x)\rangle_{\mathcal{H}^{\sigma}_{d(x)}}d\mu^{u}(xH)=|b(u)|^{2} \langle f,f\rangle(u).$$ So, $$\|bf_{\alpha}\|\leq \|b\|_{\infty}\|f_{\alpha}\|.$$ It is easy to see that the map $(f,g)\to \langle f, g\rangle: \mathcal{F}^{\sigma}(G)\times \mathcal{F}^{\sigma}(G) \to C_{0}(G^{0})$ satisfies every conditions in the definition \cite[Chapter $1$, (1.1)]{lance1995hilbert}. Also, by definition $\|f\|= \|\langle f,f\rangle\|_{C_{0}(G^{0})}^{\frac{1}{2}}$. 
    
    Suppose each fibre of $\mathcal{H}^{\sigma}$ is nonzero and $G^{0}$ is locally compact Hausdorff, we can easily see that the span of $\{\langle f,g\rangle: f,g \in \mathcal{F}^{\sigma}(G)\}$ is dense in $C_{0}(G^{0})$ by Stone Weierstrass theorem on locally compact space.
 \end{proof}
 Now, using \cite[Theorem $4.2.4$]{bos2007groupoids}, we can find a family of Hilbert spaces $\{\mathcal{F}^{u}\}_{u\in G^{0}}$ which forms a continuous Hilbert bundle $\Bar{\mathcal{F}}^{\sigma}$, also denoted as $\Bar{\mathcal{F}}^{\sigma}_{G}$, with fibres 
 $$ \mathcal{F}^{u}= \mathcal{F}^{\sigma}(G)/\mathcal{N}^{u} $$ where $$\mathcal{N}^{u}= \{f\in \mathcal{F}^{\sigma}(G): \|f\|^{2}(u)=\langle f,f\rangle(u)=0\}.$$
 We can identify $\mathcal{F}^{u}$ with the completion of $\{f_{|_{G^{u}}}: f\in \mathcal{F}_{0}^{\sigma}(G)\}$ under the norm $\sqrt{\langle f,f\rangle(u)}$.
 The space of continuous sections can be defined as 
 $$\Delta=\{t_{F} \in C_{0}(G^{0},\Bar{\mathcal{F}}^{\sigma}):t_{F}(u)= F 
 + \mathcal{N}^{u},F\in \mathcal{F}^{\sigma}(G)\}.$$
 More details can also be referred to in \cite{BSMF_1963__91__227_0,bos2007groupoids}.

 \begin{lem}
     Suppose $g \in C_{c}(G^{u},\mathcal{H}^{\sigma})$, then there exist $h\in C_{c}(G,\mathcal{H}^{\sigma})$ such that $h_{|_{G^{u}}}= g$.
 \end{lem}
 \begin{proof}
     Let $K$ be the compact support of $g$. Let $U_{0}$  be a compact neighbourhood of $K$ in $G$. For each $x\in K, \epsilon>0$, by continuity of $g$, there exist a neighborhood $U_{x} \subset {U_{0}}$ of $x$ in $G$ and a continuous compactly supported function $h_{x}$ with support in $U_{0}$ such that $\|g(y)-h_{x}(y)\|_{\mathcal{H}^{\sigma}_{d(y)}}< \epsilon$ in $U_{x}\cap G^{u}$.
     
      Since, $K$ is compact, there exist a finite cover $\{U_{x_{i}}\}_{i=1}^{n}$of $K$. Let $\{b_{i}\}$ be such that $b_{i}\in C_{c}(U_{x_{i}})$ such that $\sum_{i=1}^{n} b_{i}=1$ on $K$. Then $\sum_{i=1}^{n}b_{i}g =g$ and $\|g(y)-\sum_{i=1}^{n}b_{i}h_{x_{i}}(y)\|_{\mathcal{H}^{\sigma}_{d(y)}}<\epsilon$ for all $y\in G^{u}$.
      Now for  $\epsilon=\frac{1}{n},n\in \mathbb{N}$, there exist a sequence $\{h_{n}\}\in C_{c}(G,\mathcal{H}^{\sigma})$ supported in $U_{0}$, such that $\|g(y)-h_{n}(y)\|_{\mathcal{H}^{\sigma}_{d(y)}}< \frac{1}{n}$ for every $y \in G^{u}$. By passing through subsequence, if necessary, the sequence $\{h_{n}\}$ is such that $\|h_{n+1}^{u}-h_{n}^{u}\|_{\infty}<\frac{1}{2^{n}}$. We can define a sequence of compactly supported functions, $\{g_{n}\}$ such that $g_{n}^{u}= h_{n+1}^{u}-h_{n}^{u},n\geq 1$ and $g_{0}= h_{1}$ with $\|g_{n}\|_{\infty}< \frac{1}{2^{n}}$. Then $h=\sum_{n=0}^{\infty}g_{n}$ is the required function. 
      \end{proof}
  Now, we show that the groupoid $G$ acts on the continuous field of Hilbert space $(\Bar{\mathcal{F}}_{G}^{\sigma},\Delta)$ through left regular representation.
  \begin{prop}
      The representation $(\Bar{\mathcal{F}}_{G}^{\sigma},\Delta,L)$ is a continuous unitary representation of $G$ where $$ L(x): \mathcal{F}^{d(x)}\to \mathcal{F}^{r(x)}, ~~(L(x)f)(y)= f(x^{-1}y). $$
  \end{prop}
  \begin{proof}
      Since, for $\alpha \in C_{c}(G,\mathcal{H}^{\sigma})$ with $\alpha(x)\in \mathcal{H}^{\sigma}_{d(x)}$, $L(x_{0})\alpha \in C_{c}(G^{r(x_{0})},\mathcal{H}^{\sigma})$ and by Proposition $3.2$ and Lemma $3.4$, we can easily see that $L(x_{0})f_{\alpha} \in \mathcal{F}^{r(x_{0})}$.
      
      We prove each $L(x)$ is isometry and the map $x\to  \langle L(x)t_{f}(d(x)),t_{g}(r(x))\rangle(r(x))$ is continuous. The other conditions follow easily.
      
      With the equivariance of $\{\mu^{u}\}$, we can see that 
      
          For $f\in \mathcal{F}_{0}^{\sigma}(G)$,
          \begin{align*}
          \langle L(x_{0})f,L(x_{0})f\rangle(r(x_{0}))&=\int_{G/H}\langle f(x_{0}^{-1} y),f(x_{0}^{-1} y)\rangle _{\mathcal{H}^{\sigma}_{d(y)}}d\mu^{r(x_{0})}(yH)\\
          &= \int_{G/H}\langle f(y),f(y)\rangle _{\mathcal{H}^{\sigma}_{d(y)}}d\mu^{d(x_{0})}(yH)\\
          &=\langle f,f\rangle(d(x_{0}))
      \end{align*}
      
      Let $t_{f}$ and $t_{g}$ be two continuous sections, where $f,g \in \mathcal{F}_{0}^{\sigma}(G)$, then
      \begin{align*}
          \langle L(x)t_{f}(d(x)),t_{g}(r(x))\rangle(r(x))&=\int_{G/H}\langle f(x^{-1}\cdot y),g(y)\rangle _{\mathcal{H}^{\sigma}_{d(y)}}d\mu^{r(x)}(yH)
      \end{align*}
      The function inside the integral is a continuous function over $G*G/H$ which is a closed subset of $G\times G/H$ and hence by a similar argument of \cite[Lemma $3.12$]{bos2011continuous}, we can say that $x\to \langle L(x)t_{f}(d(x)),t_{g}(r(x))\rangle(r(x))$ is continuous.
  \end{proof}
  The continuous unitary representation $(\Bar{\mathcal{F}}_{G}^{\sigma},\Delta,L)$ is called the representation induced by $\sigma$ and can be denoted by $\left(ind^{G}_{H}(\sigma),\mu \right)$.
  
  Next, we prove an important lemma which will be useful later in this paper. 
  
   For $f\in C_{c}(G)$, $t \in C_{0}(G^{0},\mathcal{H}^{\sigma})$, we can define 
   $\mathcal{E}(f,t)\in \mathcal{F}^{\sigma}_{0}(G)$ as
   $$ \mathcal{E}(f,t)(x)=\int f(xh)\sigma(h)t(d(h))d\lambda_{H}^{d(x)}(h).$$
    
   \begin{lem}
       Let $H$ be a closed subgroupoid of locally compact groupoid $G$ and $\sigma$ be a representation of $H$, and put 
       $$ \mathcal{E}(C_{c}(G),C_{0}(G^{0},\mathcal{H}^{\sigma}))=\{\mathcal{E}(f,t): f\in C_{c}(G),t\in C_{0}(G^{0},\mathcal{H}^{\sigma})\}$$
       then,  $\mathcal{E}(C_{c}(G),C_{0}(G^{0},\mathcal{H}^{\sigma}))$ is a total subset  in $\mathcal{F}^{\sigma}(G)$. Also, for $g \in C_{c}(G),t\in C_{0}(G^{0},\mathcal{H}^{\sigma})$, 
       $$\|\mathcal{E}(g,t)\|\leq c \|g\|_{\infty}\|t\|, ~~ \text{where}~c~\text{depends only on the support of}~ g.$$
    \end{lem}
    \begin{proof}
       It is enough to show that $\mathcal{E}(C_{c}(G),C_{0}(G^{0},\mathcal{H}^{\sigma}))$ is a total subset in $\mathcal{F}^{\sigma}_{0}(G)$. Let $f_{\alpha} \in \mathcal{F}^{\sigma}_{0}(G)$, where $\alpha \in C_{c}(G, \mathcal{H}^{\sigma})$ with compact support $K$. There exist a compact neighbourhood $U_{0}$ containing $K$.

       For any $\epsilon >0$, there exist $f_{i} \in C_{c}(G)$ supported in $U_{0}$ and $t_{i}\in C_{0}(G^{0},\mathcal{H}^{\sigma}), i=1,2,...,n$, such that $\|\alpha(x)-\sum_{i=1}^{n}f_{i}(x)t_{i}(d(x))\|_{\mathcal{H}^{\sigma}_{d(x)}}<\epsilon$ for all $x\in G$.
       
       Let, $F_{n}= \sum_{i=1}^{n} \mathcal{E}(f_{i},t_{i})$. Then, $$f_{\alpha}(x)-F_{n}(x) = \int_{H} \sigma(h) [\alpha(xh)-\sum_{i=1}^{n}f_{i}(xh)t_{i}(d(h))] d\lambda_{H}^{d(x)}(h).$$
       For $x\in U_{0}$, $$ \|f_{\alpha}(x)-F_{n}(x)\|_{\mathcal{H}^{\sigma}_{d(x)}} \leq \int_{H} \|\alpha(xh)-\sum_{i=1}^{n} f_{i}(xh)t_{i}(d(h))\|d\lambda_{H}^{d(x)}(h) \leq \epsilon M,$$ where $M= \sup_{u\in G^{0}}\lambda_{H}^{u}(U_{0}^{-1}U_{0}\cap H).$ So,
        \begin{align*}
            \langle f_{\alpha}-F_{n},f_{\alpha}-F_{n} \rangle(u) &= \int_{G/H} \|f_{\alpha}(x)-F_{n}(x)\|^{2}_{\mathcal{H}^{\sigma}_{d(x)}}d\mu^{u}(xH) \leq \epsilon^{2} M^{2} \sup_{u\in G^{0}}\mu^{u}(q_{H}(U_{0})).
        \end{align*}
       Hence,  $\mathcal{E}(C_{c}(G),C_{0}(G^{0},\mathcal{H}^{\sigma}))$ is a total subset  in $\mathcal{F}^{\sigma}(G)$.
      
     Observe that, $q(supp(\mathcal{E}(g,t)))\subset q(supp(g))$ where $g \in C_{c}(G),t\in C_{0}(G^{0},\mathcal{H}^{\sigma})$. Let $supp (g)$ be contained in a relatively compact set $K$. By the property of $\mathcal{E}(g,t)$, we can see that 
     \begin{align*}
       &\left|\Big\langle \mathcal{E}(g,t)(x), ~\mathcal{E}(g,t)(x)\Big\rangle\right|\\
       \leq & \iint\limits_{H}\left|g(xh)g(xh')\Big\langle \sigma(h) (t(d(h))),\sigma(h')(t(d(h')))\Big\rangle\right| d\lambda_{H}^{d(x)}(h)d\lambda_{H}^{d(x)}(h')  \\
       \leq & c' \|g\|_{\infty}^{2}\|t\|^{2},~~~~ \text{where}~~ c'= (\sup_{u\in G^{0}}\lambda_{H}^{u}(K^{-1}K))^{2}.  
     \end{align*} 
     Since
$q_H(\operatorname{supp}\mathcal{E}(g,t))\subseteq q_H(K)$,
\begin{center}
$\displaystyle \langle\mathcal{E}(g,t),\mathcal{E}(g,t)\rangle(u)
\le c'\,\|g\|_\infty^2\|t\|^2\sup_{v\in G^0}\mu^v(q_H(K)),$
\end{center}
and the supremum is finite. So $\|\mathcal{E}(g,t)\|\le c\,\|g\|_\infty\|t\|$,
where $c$ depends only on $K$.
    \end{proof}

    \begin{rem}
      If $C_{0}(G^{0},\mathcal{H}^{\sigma})$ is countably generated, we can easily see that  $\mathcal{F}^{\sigma}(G,\mu)$ is countably generated.
  \end{rem} 

  The following proposition provides a relation between two continuous unitarily equivalent representations of closed subgroupoids and their induced representations, as well as the direct sum of representations. 
 
  \begin{thm} Let $H$ be a closed wide subgroupoid of $G$ and $\mu$ a full equivariant $r_{G^{0}}$-system on $G/H$.
  \begin{enumerate}[(i)]
      \item  If $\sigma$ and $\sigma'$ are unitarily equivalent representations of $H$, then $\left(\operatorname{ind}_{H}^{G}(\sigma),\mu\right)$ and $\left(\operatorname{ind}_{H}^{G}(\sigma'),\mu \right)$ are equivalent representations of $G$.
      \item If $\{\sigma_{i}\}_{i\in I}$ is any family of representations of $H$, then $\left(\operatorname{ind}_{H}^{G}(\bigoplus_{i\in I} \sigma_{i}),\mu \right)$ is equivalent to $\bigoplus_{i\in I} \left(ind^{G}_{H}(\sigma_{i}),\mu \right)$.
  \end{enumerate}
       
  \end{thm}
  \begin{proof}
     $(i) \Rightarrow$ Since $\sigma$ and $\sigma'$ are equivalent, there exist $T \in Mor(\sigma,\sigma')$ such that each $T_{u}$ is unitary .
      For $\alpha \in C_{c}(G,H^{\sigma})$ by continuity of $T$, the function $T\alpha$, defined as 
      $$ T\alpha(x)=T_{d(x)}\alpha(x)$$
      is a compactly supported continuous function on $G$ with range in $\mathcal H^{\sigma'}$ and we can easily see that every functions in $C_{c}(G,\mathcal{H}^{\sigma'})$ is of this form.

      For $f_{\alpha}\in \mathcal{F}_{0}^{\sigma}(G)$, define
      \begin{align*}
          Tf_{\alpha}(x)= T_{d(x)}f_{\alpha}(x)&=\int_{H}T_{d(x)}\sigma(\eta)\alpha(x\eta)d\lambda_{H}^{d(x)}(\eta)\\ &= \int_{H}\sigma'(\eta)T_{d(\eta)}\alpha(x\eta)d\lambda_{H}^{d(x)}(\eta)\\ &= \int_{H}\sigma'(\eta)T\alpha(x\eta)d\lambda_{H}^{d(x)}(\eta)
          = f_{T\alpha}(x)
      \end{align*}
      Thus, $Tf_{\alpha} \in \mathcal{F}_{0}^{\sigma'}(G)$.
  
  Also, it is easy to see that $T$ preserves the $C_{0}(G^{0})$-valued inner product.
  Note that $T$ is also a module map. Hence, by \cite[Theorem $3.5$]{lance1995hilbert}, $\mathcal{F}^{\sigma}(G)$ and $\mathcal{F}^{\sigma'}(G)$ are unitarily equivalent  Hilbert $C_{0}(G^{0})$-modules and corresponding Hilbert spaces $\mathcal{F}^{u}$ and $\mathcal{F}'^{u}$ are isometrically isomorphic.  The continuous map  induced by $T$,  denoted by $T':\Bar{\mathcal{F}}_{G}^{\sigma} \to \Bar{\mathcal{F}}_{G}^{\sigma'}$ intertwines with the induced representations. Let $\rho =\left(\operatorname{ind}_{H}^{G}(\sigma),\mu\right), \rho' =\left(\operatorname{ind}_{H}^{G}(\sigma'),\mu\right)$ and $f\in \mathcal{F}_{0}^{\sigma}(G)$,
  \begin{align*}
     T'_{r(x)}(\rho(x)f)(y) = T_{d(y)}(f(x^{-1}y))= (T'_{d(x)}f)(x^{-1}y)= \rho'(x)(T'_{d(x)}f)(y).
  \end{align*}
   Hence, $\left(\operatorname{ind}_{H}^{G}(\sigma),\mu\right)$ and $\left(\operatorname{ind}_{H}^{G}(\sigma'),\mu\right)$ are equivalent representations of $G$.
   
   $(ii) \Rightarrow$ Let $\sigma=\bigoplus\limits_{i=I} \sigma_{i} ,\pi=\operatorname{ind}_{H}^{G}\sigma$ and $\pi_{i}=\operatorname{ind}_{H}^{G}\sigma_{i},i\in I$. $\mathcal{H}^{\sigma}$ is the direct sum of Hilbert bundles $\mathcal{H}^{\sigma_{i}}$ and $p^{i}:\mathcal{H}^{\sigma}\to \mathcal{H}^{\sigma_{i}} $ be the projection as per the definition given in \cite[Section $15.14$]{fell1988representations}. For $f_{\alpha} \in \mathcal{F}^{\sigma}_{0}(G)$ and $i\in I$, define $f_{\alpha}^{i}:G\to \mathcal{H}^{\sigma_{i}}$ as $f_{\alpha}^{i}(x)=p^{i}(f_{\alpha}(x))$. So for $h\in H$ and $x\in G$,
   $$f_{\alpha}^{i}(xh)=p^{i}(f_{\alpha}(xh))=p^{i}(\sigma(h^{-1})f_{\alpha}(x))=\sigma_{i}(h^{-1})f^{i}_{\alpha}(x).$$
Note that $q(supp (f_{\alpha}^{i})\subset q(supp(f_{\alpha}))$. Hence $f_{\alpha}^{i}\in \mathcal{F}^{\sigma_{i}}(G)$. 

Define $W: \mathcal{F}_{0}^{\sigma}(G)\to \bigoplus\limits_{i\in I}\mathcal{F}^{\sigma_{i}}(G)$ as $(Wf_{\alpha})_{i}(x)=f_{\alpha}^{i}(x)$ for $x \in G$.
   \begin{align*}
       \langle Wf_{\alpha},Wf_{\alpha}\rangle(u) &= \sum\limits_{i\in I}\Big\langle (Wf_{\alpha})_{i},(Wf_{\alpha})_{i}\Big\rangle(u) \\ &= \sum\limits_{i\in I}\int_{G/H}\Big\langle p^{i}f_{\alpha}(x),p^{i}f_{\alpha}(x)\Big\rangle d\mu^{u}(xH)\\&=\int_{G/H}\sum\limits_{i\in I}\Big\langle p^{i}f_{\alpha}(x),p^{i}f_{\alpha}(x)\Big\rangle d\mu^{u}(xH)\\&=\int_{G/H}\Big\langle f_{\alpha}(x),f_{\alpha}(x)\Big\rangle d\mu^{u}(xH)=\langle f_{\alpha},f_{\alpha}\rangle(u).
   \end{align*} 
   So, $W$ is an isometric module map. To prove image of $W$ is dense, let $\Lambda$ be a finite set in $I$ and $\eta_{i} \in \mathcal{F}_{0}^{\sigma_{i}}(G),i\in I$, define $\eta:G\to  \mathcal{H}^{\sigma}$ as $(\eta(x))_{i}= \eta_{i}(x)$ when $i\in \Lambda$, else $(\eta(x))_{i}=0$ . Clearly, $\eta \in \mathcal{F}^{\sigma}_{0}$ and $W\eta =\bigoplus\limits_{i\in I} \eta'_{i}$, where $\eta'_{i}=\eta_{i}$ when $i\in \Lambda$, else $\eta'_{i}\equiv 0$. By density of $\bigoplus\limits_{i\in I}\mathcal{F}_{0}^{\sigma_{i}}(G)$ in $\bigoplus\limits_{i\in I}\mathcal{F}^{\sigma_{i}}(G)$, result follows. Now, it is clear that $\mathcal{F}^{\sigma}(G)$ and  $\bigoplus\limits_{i\in I}\mathcal{F}^{\sigma_{i}}(G)$ are unitarily equivalent modules and hence the corresponding  Hilbert bundles are equivalent. Note that the morphism induced by $W$ intertwines with the induced representations.
   \end{proof} 
   Let $(\mathcal{H}^\sigma,\sigma)$ be a continuous unitary representation of a
groupoid. For a Hilbert space $V$, let $\overline{V}$ denote its conjugate
space: it has the elements $\bar v$, $v\in V$, with $\bar v+\bar w=\overline{v+w}$,
$a\bar v=\overline{\bar a v}$ and $\langle\bar v,\bar w\rangle=\overline{\langle v,w\rangle}$.
The conjugate representation $(\mathcal{H}^{\bar\sigma},\bar\sigma)$ acts on
the continuous field $\{\overline{\mathcal{H}^\sigma_u}\}_{u\in G^0}$ by
$\bar\sigma(x)\bar v=\overline{\sigma(x)v}$ for $v\in\mathcal{H}^\sigma_{d(x)}$.
Similarly, if $\Delta$ is a Hilbert $C_0(G^0)$-module, the conjugate module
$\overline{\Delta}=\{\bar\xi:\xi\in\Delta\}$ has the inner product
$\langle\bar\xi,\bar\eta\rangle(u)=\overline{\langle\xi,\eta\rangle(u)}$ and the
action $a\cdot\bar\xi=\overline{\bar a\,\xi}$ for $a\in C_0(G^0)$. The fibres of
the continuous field corresponding to $\overline{\Delta}$ are the conjugate
spaces of the fibres corresponding to $\Delta$.

\begin{thm}
Let $\sigma$ be a continuous unitary representation of a closed wide
subgroupoid $H$. Then $\overline{\operatorname{ind}_H^G(\sigma,\mu)}$ is
unitarily equivalent to $\operatorname{ind}_H^G(\bar\sigma,\mu)$.
\end{thm}

\begin{proof}
For $f\in\mathcal{F}_0^\sigma(G)$, define $\tilde f(x)=\overline{f(x)}$, so that
$\tilde f(x)\in\overline{\mathcal{H}^\sigma_{d(x)}}$. For $\xi\in H$ with
$r(\xi)=d(x)$,
$\tilde f(x\xi)=\overline{\sigma(\xi^{-1})f(x)}=\bar\sigma(\xi^{-1})\tilde f(x)$,
so $\tilde f\in\mathcal{F}_0^{\bar\sigma}(G)$. The map
$W\colon\overline{\mathcal{F}_0^\sigma(G)}\to\mathcal{F}_0^{\bar\sigma}(G)$,
$W\bar f=\tilde f$, is a module map, and
\begin{center}
$\displaystyle \langle W\bar f,W\bar g\rangle(u)
=\int_{G/H}\overline{\langle f(x),g(x)\rangle}\,d\mu^u(xH)
=\overline{\langle f,g\rangle(u)}=\langle\bar f,\bar g\rangle(u).$
\end{center}
Every element of $\mathcal{F}_0^{\bar\sigma}(G)$ is of the form $\tilde f$, so
$W$ extends to a unitary map of Hilbert $C_0(G^0)$-modules. Finally, for $x\in G$
and $f\in\mathcal{F}_0^\sigma(G)$,
\begin{center}
$\displaystyle \big(W\,\overline{\operatorname{ind}_H^G(\sigma,\mu)(x)f}\big)(y)
=\overline{f(x^{-1}y)}
=\big(\operatorname{ind}_H^G(\bar\sigma,\mu)(x)\,W\bar f\big)(y),$
\end{center}
so the map of Hilbert bundles induced by $W$ intertwines the two
representations.
\end{proof}

      Next we show that for \'etale groupoids a full equivariant system $\mu$ always
exists, and that any two such systems $\mu$ and $\mu'$ differ by a continuous
function $c$, that is, $\mu'=c\,\mu$. As a consequence,
$\operatorname{ind}_{H}^{G}(\sigma,\mu)$ is unitarily equivalent to
$\operatorname{ind}_{H}^{G}(\sigma,\mu')$.

\begin{prop}
Let $G$ be an \'etale groupoid and let $H$ be a closed wide subgroupoid of $G$
with a Haar system. Then there exists a full equivariant $r_{G^{0}}$-system
$\mu$ on $G/H$. If $\mu'$ is another full equivariant $r_{G^{0}}$-system, then
$\mu'=c\,\mu$ for a continuous $G$-invariant function
$c\colon G/H\to(0,\infty)$. Moreover, $\operatorname{ind}_{H}^{G}(\sigma,\mu)$
is unitarily equivalent to $\operatorname{ind}_{H}^{G}(\sigma,\mu')$.
\end{prop}

\begin{proof}
Since $G$ is \'etale, $G$ has a basis $\{U_{n}\}$ of open bisections. The
quotient map $q_{H}$ is open, so $\{q_{H}(U_{n})\}$ is a basis of $G/H$. If
$x_1H=x_2H$ with $x_1,x_2\in U_n$, then $r(x_1)=r(x_2)$, so $x_1=x_2$. Hence
$q_H$ is injective on $U_n$, and $r_{G^0}$ restricted to $q_H(U_n)$ is the
homeomorphism $r|_{U_n}\circ(q_H|_{U_n})^{-1}$ onto the open set $r(U_n)$.
Thus $r_{G^{0}}$ is a local homeomorphism, and each fibre
$r_{G^{0}}^{-1}(u)$ is discrete.

Let $\mu^{u}$ be the counting measure on $r_{G^{0}}^{-1}(u)$. Its support is
the whole fibre. For $\varphi\in C_c(G/H)$, the map
$u\mapsto\sum_{r_{G^0}(p)=u}\varphi(p)$ is continuous, since $r_{G^{0}}$ is a
local homeomorphism. For $x\in G$, the map $p\mapsto x\cdot p$ is a bijection
from $r_{G^{0}}^{-1}(d(x))$ onto $r_{G^{0}}^{-1}(r(x))$, so $\mu$ is
equivariant. Hence $\mu$ is a full equivariant $r_{G^0}$-system.

Let $\mu'$ be another full equivariant $r_{G^0}$-system. Since each fibre is
discrete, $\mu'^{u}=\sum_{r_{G^0}(p)=u}c(p)\,\delta_p$, where
$c(p)=\mu'^{u}(\{p\})$. Here $0<c(p)<\infty$, since $\mu'^u$ is a Radon measure
with support $r_{G^0}^{-1}(u)$. So $\mu'=c\,\mu$.

To see that $c$ is continuous, let $p\in G/H$, and let $U$ be an open set
containing $p$ such that $r_{G^{0}}|_{U}$ is a homeomorphism onto an open set.
Let $s=(r_{G^{0}}|_{U})^{-1}$. Let $V\subseteq U$ be a compact neighbourhood of
$p$, and choose $\Phi\in C_{c}(G/H)$ with $\Phi\equiv1$ on $V$ and
$\operatorname{supp}\Phi\subseteq U$. Since $U$ meets each fibre in at most
one point,
\begin{align*}
\int_{G/H}\Phi(xH)\,d\mu'^{u}(xH)
=\begin{cases}
\Phi(s(u))\,c(s(u)), & u\in r_{G^{0}}(U),\\
0, & \text{otherwise}.
\end{cases}
\end{align*}
The left side is continuous in $u$. For $u\in r_{G^0}(\operatorname{int}V)$
we have $\Phi(s(u))=1$, so $u\mapsto c(s(u))$ is continuous there. Since
$c(q)=c(s(r_{G^0}(q)))$ for $q\in\operatorname{int}V$, the function $c$ is
continuous near $p$.

The function $c$ is $G$-invariant. Indeed, by the equivariance of $\mu$ and
$\mu'$, for $x\in G$ and $\varphi\in C_c(G/H)$,
\begin{align*}
\sum_{r_{G^0}(p)=d(x)}\big(c(p)-c(x\cdot p)\big)\,\varphi(x\cdot p)=0 .
\end{align*}
Choosing $\varphi$ supported in a set on which $r_{G^0}$ is injective, with
$\varphi(x\cdot p_0)=1$, gives $c(p_0)=c(x\cdot p_0)$.

Finally, define
$W\colon\mathcal{F}^{\sigma}(G,\mu')\to\mathcal{F}^{\sigma}(G,\mu)$ by
$Wf(x)=\sqrt{c(q_{H}(x))}\,f(x)$. Since $c\circ q_H$ is constant on cosets,
$Wf$ satisfies the covariance condition. Since $\mu'=c\,\mu$, $W$ preserves the
$C_0(G^0)$-valued inner products, and $W$ is a module map with inverse
$f\mapsto(c\circ q_H)^{-1/2}f$. So $W$ is a unitary map of Hilbert
$C_0(G^0)$-modules. Since $c$ is $G$-invariant, $W$ commutes with the left
translations $L(x)$, so the induced map of Hilbert bundles intertwines
$\operatorname{ind}_{H}^{G}(\sigma,\mu')$ and
$\operatorname{ind}_{H}^{G}(\sigma,\mu)$.
\end{proof}

       \begin{rem}
The proof of Proposition 3.10 shows more. Let $\mu$ and $\mu'$ be full
equivariant $r_{G^0}$-systems on $G/H$ with $\mu'=c\,\mu$, where
$c\colon G/H\to(0,\infty)$. Since both systems are equivariant, for each
$x\in G$ we have $c(x\cdot p)=c(p)$ for $\mu^{d(x)}$-almost every $p$, and for
every $p$ if $c$ is continuous. If $c$ is continuous, or more generally if $c$
is a Borel function bounded on compact subsets such that $\mu^u(N)=0$ for all $u\in G^0$, where
$N$ is the set of points of discontinuity of $c$,  then
$\operatorname{ind}_H^G(\sigma,\mu)$ is unitarily equivalent to
$\operatorname{ind}_H^G(\sigma,\mu')$. Indeed, the map
$Wf=\sqrt{c\circ q_H}\,f$ preserves the norm on each fibre, intertwines the
two representations by the invariance of $c$, and maps continuous sections to
continuous sections by \cite[Corollary 8.4.2]{Bogachev2007}, since only $\sqrt{c}$
appears in the inner products with continuous sections.
\end{rem}
 We now give some examples of induced representations. Example $3$ shows that,
unlike the group case, the induced representation can depend on the choice
of the full equivariant system $\mu$.
\subsection{Examples}
 \begin{exm}
     Let $G$ be a locally compact second countable groupoid with a left Haar system $\{\lambda^u\}_{u\in G^0}$.

\smallskip
\textbf{The subgroupoid.}
Let $H=G^0$. Since $G$ is Hausdorff, $G^0$ is closed in $G$, so $H$ is a closed wide subgroupoid of $G$. Its Haar system is $\{\delta_u\}_{u\in G^0}$. Let $\sigma$ be the trivial representation of $H$ on the line bundle $G^0\times\mathbb{C}$, that is, $\sigma(u)=1$ for all $u\in G^0$.

\smallskip
\textbf{The space $G/H$ and the measure system.}
For $x\in G$, the coset $xH$ is $\{x\,d(x)\}=\{x\}$. So $G/H=G$, the quotient map $q_H$ is the identity, and the moment map is $r$. The action of $G$ on $G/H$ is left multiplication. A full equivariant $r_{G^0}$-system on $G$ is therefore exactly a left Haar system. We take
\begin{center}
$\displaystyle
\mu^u=\lambda^u,\qquad u\in G^0 .
$
\end{center}

\smallskip
\textbf{The fibres.}
For $\xi\in H$, the condition $f(x\xi)=\sigma(\xi^{-1})f(x)$ reads $f(x\,d(x))=f(x)$, which always holds. The condition that $q_H(\mathrm{supp}\,f)$ is compact means that $\mathrm{supp}\,f$ is compact. Hence
\begin{center}
$\displaystyle
F_0^\sigma(G)=C_c(G),\qquad
\langle f,g\rangle(u)=\int_{G^u} f(x)\,\overline{g(x)}\,d\lambda^u(x).
$
\end{center}
So the fibre $F^u$ is the completion of $\{f|_{G^u}: f\in C_c(G)\}$ in the norm of $L^2(G^u,\lambda^u)$, that is,
\begin{center}
$\displaystyle
F^u=L^2(G^u,\lambda^u).
$
\end{center}

\smallskip
\textbf{The action.}
For $x\in G$, the operator $L(x)\colon L^2(G^{d(x)},\lambda^{d(x)})\to L^2(G^{r(x)},\lambda^{r(x)})$ is
\begin{center}
$\displaystyle
(L(x)f)(y)=f(x^{-1}y),\qquad y\in G^{r(x)} .
$
\end{center}
Thus $\operatorname{ind}_{G^0}^{G}(1,\lambda)$ is the left regular representation of $G$.
\end{exm}

\begin{rem}
This is the groupoid version of the fact that the left regular representation of a locally compact group is induced from the trivial representation of the trivial subgroup.
\end{rem}
 
\begin{exm}\label{ex:rotation}
Fix $\theta\in\mathbb{R}$ and let $\alpha\colon\mathbb{T}\to\mathbb{T}$, $\alpha(z)=e^{2\pi i\theta}z$.

\smallskip
\noindent\emph{The groupoids.}
Let $G=\mathbb{Z}\ltimes\mathbb{T}$ be the transformation groupoid. Its elements are pairs $(n,z)$, with
\begin{center}
$\displaystyle
d(n,z)=z,~ r(n,z)=\alpha^n z,~ (m,\alpha^n z)(n,z)=(m+n,z),~ (n,z)^{-1}=(-n,\alpha^n z).
$
\end{center}
Let $H=2\mathbb{Z}\ltimes\mathbb{T}$. Then $H$ is a closed wide subgroupoid of $G$. Both $G$ and $H$ are \'etale, and we use counting measures as Haar systems.

\smallskip
\noindent\emph{The space $G/H$ and the measure system.}
The coset of $(n,z)$ is $(n,z)H=\{(n+2m,\alpha^{-2m}z): m\in\mathbb{Z}\}$. The map
\begin{center}
$\displaystyle
(n,z)H\longmapsto (n \bmod 2,\ \alpha^n z)
$
\end{center}
is a homeomorphism of $G/H$ onto $\mathbb{Z}_2\times\mathbb{T}$. The moment map becomes the second coordinate, and the action of $G$ becomes $(k,u)\cdot(j,u)=(j+k,\ \alpha^k u)$. Let
\begin{center}
$\displaystyle
\mu^u=(\delta_0+\delta_1)\times\delta_u ,\qquad u\in\mathbb{T} .
$
\end{center}
This is a full equivariant $r_{G^0}$-system on $G/H$, since counting measure on $\mathbb{Z}_2$ is invariant under translation.

\smallskip
\textbf{The representation $\sigma$.}
Let $w\colon\mathbb{T}\to\mathbb{T}$ be continuous. On the trivial line bundle $\mathbb{T}\times\mathbb{C}$, define $\sigma(2,z)=w(z)$ and extend by multiplicativity:

\begin{align*}
    &\sigma(2m,z)=w(\alpha^{2(m-1)}z)\cdots w(\alpha^{2}z)\,w(z)\quad (m>0),\\
&\sigma(-2m,\alpha^{2m}z)=\sigma(2m,z)^{-1}.
\end{align*}

Then $\sigma$ is a continuous unitary representation of $H$.

\smallskip
\textbf{The fibres.}
Fix $u\in\mathbb{T}$. Then $G^u=\{(n,\alpha^{-n}u): n\in\mathbb{Z}\}$. For $f\in F_0^\sigma(G)$, write $\varphi_u(n)=f(n,\alpha^{-n}u)$. For $x=(n,\alpha^{-n}u)$ and $\xi=(2m,\alpha^{-(n+2m)}u)\in H$, we have $x\xi=(n+2m,\alpha^{-(n+2m)}u)$ and $\xi^{-1}=(-2m,\alpha^{-n}u)$. So the condition $f(x\xi)=\sigma(\xi^{-1})f(x)$ becomes
\begin{center}
$\displaystyle
\varphi_u(n+2m)=\sigma(-2m,\alpha^{-n}u)\,\varphi_u(n),\qquad n,m\in\mathbb{Z} .
$
\end{center}
Hence $\varphi_u$ is determined by $\varphi_u(0)$ and $\varphi_u(1)$, and these two values can be chosen freely. Since $\sigma$ is unitary,
\begin{center}
$\displaystyle
\langle f,f\rangle(u)=|\varphi_u(0)|^2+|\varphi_u(1)|^2 .
$
\end{center}
Thus $F^u\cong\mathbb{C}^2$ via $f\mapsto(\varphi_u(0),\varphi_u(1))$, and the induced bundle $\bar F^\sigma$ is the trivial bundle $\mathbb{T}\times\mathbb{C}^2$.

\smallskip
\textbf{The action.}
Let $x\in\mathbb{T}$ and $f\in F^x$. A direct computation gives $(L(1,x)f)(n,\alpha^{-n}\alpha x)=\varphi_x(n-1)$. By the covariance condition, $\varphi_x(-1)=\sigma(2,\alpha^{-1}x)\varphi_x(1)=w(\alpha^{-1}x)\varphi_x(1)$. So, in the bases above,
\begin{center}
$\displaystyle
\operatorname{ind}_H^G(\sigma,\mu)(1,x)=\begin{pmatrix}0 & w(\alpha^{-1}x)\\ 1 & 0\end{pmatrix}\colon F^x\to F^{\alpha x}.
$
\end{center}
\end{exm}

Every element of $G$ is a product of elements $(1,y)$ and their inverses, so this formula determines the whole representation.  
\begin{rem}
 If $w\equiv t$ is constant, then $\sigma$ comes from the character $2m\mapsto t^m$ of $2\mathbb{Z}$. Then $\operatorname{ind}_{H}^{G}(\sigma,\mu)(1,x)=\begin{pmatrix}0&t\\1&0\end{pmatrix}$ for all $x$, which is the classical representation $\operatorname{ind}_{2\mathbb{Z}}^{\mathbb{Z}}\chi_t$.
\end{rem}
\begin{exm}
Let $X=[0,1]$ and let $G=X\times\mathbb{R}\times X$, with
$(x,s,y)(y,t,z)=(x,s+t,z)$, $r(x,s,y)=x$ and $d(x,s,y)=y$. Then $G$ has the Haar
system $\lambda^x=\delta_x\times ds\times dy$. Let
$H=\{(y,s,y): y\in X,\ s\in\mathbb{R}\}$ be the isotropy bundle. It is a closed
wide subgroupoid of $G$, with Haar system $\delta_y\times ds$. On the trivial
line bundle $X\times\mathbb{C}$, define
\begin{center}
$\displaystyle \sigma(y,s,y)=e^{2\pi i ys}.$
\end{center}

The coset of $(x,s,y)$ is $\{x\}\times\mathbb{R}\times\{y\}$, so
$G/H\cong X\times X$, with moment map $(x,y)\mapsto x$ and action
$(x',s,x)\cdot(x,y)=(x',y)$. For every finite Radon measure $m$ on $X$ with full
support, $\mu_m^u=\delta_u\times m$ is a full equivariant $r_{G^0}$-system on
$G/H$.

For $f\in F_0^\sigma(G)$, the covariance condition gives
$f(u,s,y)=e^{-2\pi i ys}f(u,0,y)$. Hence $F^u\cong L^2(X,m)$ via
$f\mapsto\varphi(y)=f(u,0,y)$, and the isotropy group $G^u_u\cong\mathbb{R}$
acts by
\begin{center}
$\displaystyle \big(\operatorname{ind}_H^G(\sigma,\mu_m)(u,t,u)\varphi\big)(y)=e^{2\pi i yt}\varphi(y).$
\end{center}

Let $m$ be the Lebesgue measure and
$m'=m+\sum_{n\ge1}2^{-n}\delta_{q_n}$, where $\{q_n\}$ is an enumeration of
$\mathbb{Q}\cap[0,1]$. Then $\operatorname{ind}_H^G(\sigma,\mu_m)$ and
$\operatorname{ind}_H^G(\sigma,\mu_{m'})$ are not unitarily equivalent. Indeed,
$1_{\{q_1\}}$ is a nonzero vector in $L^2(X,m')$, and $t$ acts on it as
multiplication by $e^{2\pi i q_1 t}$. A unitary intertwiner would give a
nonzero $\psi\in L^2(X,m)$ with $(e^{2\pi i(y-q_1)t}-1)\psi(y)=0$ almost
everywhere, for every $t\in\mathbb{R}$. This forces $\psi=0$ in $L^2(X,m)$, a contradiction.
\end{exm}
   \section{Induction in Stages}
   Let  $K$ and $H$ be closed wide subgroupoids such that $K\subset H \subset G$. Let $\{\mu_{G}^{u}\}_{u\in G^{0}}$  and $\{\mu_{H}^{u}\}_{u\in G^{0}}$ be full  $r_{G^{0}}$- systems over $G/H$ and $H/K$ respectively.  Then there exist a full $r_{G^{0}}$-system $\{\gamma^{u}\}_{u\in G^{0}}$ on $G/K$ as :
   $$ \int_{G/K} f d\gamma^{u}(xK)= \int_{G/H}\int_{H/K}f(xhK)d\mu^{d(x)}_{H}(hK)d\mu_{G}^{u}(xH).$$ 
   Let  $\sigma$ be a continuous unitary representation of $K$. We prove that $\left(\operatorname{ind}_{K}^{G}(\sigma), \gamma \right)$ is unitarily equivalent to $\left(\operatorname{ind}_{H}^{G}\left(\operatorname{ind}_{K}^{H}(\sigma),\mu_{H}\right),\mu_{G}\right)$. Let $\mathcal{F}_{0}^{\sigma}(H^{u})$ be the restriction of functions in $\mathcal{F}_{0}^{\sigma}(H)$ to $G^{u}\cap H= H^{u}$.
   \begin{lem}
       For $\xi \in \mathcal{F}_{0}^{\sigma}(G)$, $x\in G$, define $\Phi\xi(x) \in \mathcal{F}_{0}^{\sigma}(H^{d(x)})$ as:
       $$ (\Phi\xi(x))(h)= \xi(xh),~~ h\in H^{d(x)}.$$
       Then the mapping $\Phi\xi: x\to \Phi\xi(x)$ is contained in $\mathcal{F}_{0}^{\operatorname{ind}_{K}^{H}}(G)$.
   \end{lem}
   \begin{proof}
       Since $\xi \in \mathcal{F}_{0}^{\sigma}(G)$, by extending a suitable function in $C_{c}(H^{d(x)},\mathcal{H}^{\sigma})$ to $H$ using Lemma $3.4$ , we can easily see that $\Phi\xi(x) \in \mathcal{F}_{0}^{\sigma}(H^{d(x)})$. Let $\rho= \operatorname{ind}_{K}^{H}(\sigma)$, then
\begin{align*}
   \Phi\xi(xh_{0})(h)=\xi(xh_{0}h)=\Phi\xi(x)(h_{0}h)= \rho(h_{0}^{-1})(\Phi\xi(x))(h)
\end{align*}
   We can easily see that $supp(\Phi\xi)$ is in $q_{H}(C)$ where $C$ is a compact set such that $q_{K}(C)= supp(\xi)$.  
Suppose $\Bar{\mathcal{F}}^{\sigma}_{H}$ denotes  the continuous field of Hilbert spaces over which representation $\rho$ acts. For $F\in \mathcal{F}_{0}^{\sigma}(H)$, $t_{F}:G^{0} \to \Bar{\mathcal{F}}^{\sigma}_{H}, t_{F}(u)= F_{|_{H^{u}}}$, forms a dense set in  $C_{0}(G^{0},\Bar{\mathcal{F}}^{\sigma}_{H})$.
\begin{align*}
 \langle \Phi\xi(x),t_{F}(d(x))\rangle= \int_{H/K} \langle\Phi\xi(x)(h),F(h)\rangle d\mu^{d(x)}_{H}(hK) =  \int_{H/K} \langle\xi(xh),F(h)\rangle d\mu^{d(x)}_{H}(hK)
\end{align*}
   Using the continuity of $\xi, F$ and arguments similar to \cite[Lemma 3.2]{bos2011continuous}, we can say $x\to \langle \Phi\xi(x),t_{F}(d(x))\rangle $ is continuous. Similarly,
   $$  \langle \Phi\xi(x),\Phi\xi(x)\rangle = \int_{H/K} \langle \xi(xh),\xi(xh)\rangle d\mu^{d(x)}_{H}(hK)$$
    is continuous.
    Hence, $\Phi\xi \in \mathcal{F}_{0}^{\operatorname{ind}_{K}^{H}}(G)$.
   \end{proof}
   \begin{lem}
       The module map, $\Phi: \mathcal{F}_{0}^{\sigma}(G) \to \mathcal{F}_{0}^{\operatorname{ind}_{K}^{H}}(G)$ is an isometry.
   \end{lem}
   \begin{proof}
   \begin{align*}
      \langle \Phi\xi,\Phi\xi \rangle(u)&= \int_{G/H} \langle \Phi\xi(x),\Phi\xi(x) \rangle d\mu_{G}^{u}(xH)\\
      &=  \int_{G/H}\int_{H/K} \langle \xi(xh),\xi(xh) \rangle d\mu_{H}^{d(x)}(hK)d\mu_{G}^{u}(xH)\\
      &= \int_{G/K}\langle \xi(x),\xi(x) \rangle d\gamma^{u}(xK)= \langle \xi,\xi \rangle(u). \qedhere
   \end{align*}

   \end{proof}
   \begin{lem}
    $\Phi(\mathcal{F}_{0}^{\sigma}(G))$ is dense in $\mathcal{F}_{0}^{\operatorname{ind}_{K}^{H}}(G)$.
   \end{lem}
   \begin{proof}
      Let $\rho= \operatorname{ind}_{K}^{H}(\sigma)$. Let $f_{1} \in C_{c}(H), f_{2} \in C_{c}(G), \xi \in C_{0}(G^{0},\mathcal{H}^{\sigma}) $, take
      $$ \mathcal{E}(f_{1},\xi)(h)=\int_{K}f_{1}(hk)\sigma(k)\xi(d(k))d\lambda_{K}^{d(h)}(k)$$

      The function $\mathcal{E}(f_{1},\xi)\in \mathcal{F}_{0}^{\sigma}(H)$ and such elements form a total subset according to Lemma $3.6$.
          
       Define $f\in C_{c}(G)$ by 
       $$ f(x)= \int_{H} f_{1}(h^{-1})f_{2}(xh)d\lambda_{H}^{d(x)}(h).$$
       Then $\mathcal{E}(f,\xi) \in \mathcal{F}^{\sigma}_{0}(G),$ where $\mathcal{E}(f,\xi)(x)= \int_{K} f(xk)\sigma(k)\xi(d(k))d\lambda_{K}^{d(x)}(k)$.  
       Let $\mathcal{E}(f_{1},\xi)= F$, then
       \begin{align*}
           \mathcal{E}(f_{2},t_{F})(x)(h_{0})&=\left(\int_{H}f_{2}(xh)\rho(h)t_{F}(d(h))d\lambda_{H}^{d(x)}(h)\right)(h_{0})\\
           &= \int_{H}f_{2}(xh)\epsilon_{1}(f_{1},\xi)(h^{-1}h_{0})d\lambda_{H}^{d(x)}(h)\\
           &= \int_{H}\int_{K}f_{2}(xh)f_{1}(h^{-1}h_{0}k)\sigma(k)\xi(d(k))d\lambda_{K}^{d(h_{0})}(k)d\lambda_{H}^{d(x)}(h)
       \end{align*}
       The function $\mathcal{E}(f_{2},t_{F}) \in \mathcal{F}_{0}^{\operatorname{ind}_{K}^{H}}(G)$.
       
       Now,
       \begin{align*}
           \Phi(\mathcal{E}(f,\xi))(x)(h_{0})&=\mathcal{E}(f,\xi)(xh_{0})= \int_{K} f(xh_{0}k)\sigma(k)\xi(d(k))d\lambda_{K}^{d(h_{0})}(k)\\
           &= \int_{K}\int_{H}f_{1}(h^{-1})f_{2}(xh_{0}kh)\sigma(k)\xi(d(k))d\lambda_{H}^{d(k)}(h)d\lambda_{K}^{d(h_{0})}(k)\\
           &=\int_{K}\int_{H}f_{1}(h^{-1}h_{0}k)f_{2}(xh)\sigma(k)\xi(d(k))d\lambda_{H}^{r(h_{0})}(h)d\lambda_{K}^{d(h_{0})}(k)\\
           &= \int_{H}\int_{K}f_{1}(h^{-1}h_{0}k)f_{2}(xh)\sigma(k)\xi(d(k)) d\lambda_{K}^{d(h_{0})}(k)d\lambda_{H}^{d(x)}(h)\\
           &=\mathcal{E}(f_{2},t_{F})(x)(h_{0})
       \end{align*} 
       By Lemma $3.6$ again, functions of this form span a dense subspace of
$\mathcal{F}^{\rho}(G)$, where $\rho=\operatorname{ind}_{K}^{H}(\sigma)$. Since
they all lie in the image of $\Phi$, the image of $\Phi$ is dense, and the
result follows.
   \end{proof}
   Using all the series of lemmas above, $\Phi$ extends uniquely to a surjective isometric module map from Hilbert module $\mathcal{F}^{\sigma}(G)$ to $\mathcal{F}^{\operatorname{ind}_{K}^{H}(\sigma)}(G)$. By \cite[Theorem $3.5$]{lance1995hilbert}, $\mathcal{F}^{\sigma}(G)$ and $\mathcal{F}^{\operatorname{ind}_{K}^{H}(\sigma)}(G)$ are unitarily equivalent. Also, $\Phi$ induces a morphism between the corresponding continuous field of Hilbert spaces $\Bar{\mathcal{F}}^{\sigma}_{G}$ and $\Bar{\mathcal{F}}_{G}^{\operatorname{ind}_{K}^{H}(\sigma)}$ which intertwines with the representations  $\left(\operatorname{ind}_{K}^{G}(\sigma), \gamma \right)$ and  $\left(\operatorname{ind}_{H}^{G}\left(\operatorname{ind}_{K}^{H}(\sigma),\mu_{H}\right),\mu_{G}\right)$.
    Thus the following theorem is proved.
    \begin{thm}
        Let $K$ and $H$ be closed wide subgroupoids of $G$ such that $K\subseteq H$, and $\sigma$ be a continuous unitary representation of $K$. Then the representations $\left(\operatorname{ind}_{K}^{G}(\sigma), \gamma \right)$ and $\left(\operatorname{ind}_{H}^{G}\left(\operatorname{ind}_{K}^{H}(\sigma),\mu_{H}\right),\mu_{G}\right)$ are unitarily equivalent.
    \end{thm}
   \section{Tensor products of induced representation}
   Suppose $G_{1}$ and $G_{2}$ are two locally compact second countable groupoids, $G_{1}\times G_{2}$ is again a groupoid having unit space $ G_{1}^{0}\times G_{2}^{0}$. The product and inverse are component-wise. If $\{\lambda^{u_{1}}\}_{u_{1}\in G_{1}^{0}}$ and $\{\lambda^{u_{2}}\}_{u_{2}\in G_{2}^{0}}$ are the Haar system of $G_{1}$ and $G_{2}$ respectively, then $\{\lambda^{u_{1}}\times \lambda^{u_{2}}\}_{(u_{1},u_{2})\in G_{1}^{0}\times G_{2}^{0}}$ forms a Haar system of $G_{1}\times G_{2}$.

   If $(\mathcal{H}^{1},\pi_{1})$ and $(\mathcal{H}^{2},\pi_{2})$ are two representations of $G_{1}$ and $G_{2}$ respectively, we can define outer tensor product $\pi_{1}\times \pi_{2}$ of $G_{1} \times G_{2}$ acting on $\mathcal{H}^{1}\otimes \mathcal{H}^{2}$ such that
   
   For $(x,y)\in G_{1}\times G_{2}$, $(\pi_{1}\times \pi_{2})(x,y): \mathcal{H}^{1}_{d(x)}\otimes \mathcal{H}^{2}_{d(y)} \to \mathcal{H}^{1}_{r(x)}\otimes \mathcal{H}^{2}_{r(y)} $ 
   $$ (\pi_{1}\times \pi_{2})(x,y)v_{1}\otimes v_{2}=\pi_{1}(x)v_{1} \otimes  \pi_{2}(y)v_{2}. $$
   In this section, we prove the groupoid version of Mackey's tensor product theorem.

   Let $\sigma_{1}$ and $\sigma_{2}$ be two representations of closed wide subgroupoids $H_{1}$ and $H_{2}$ respectively. The exterior tensor product $\mathcal{F}^{\sigma_{1}}(G_{1})\otimes \mathcal{F}^{\sigma_{2}}(G_{2})$ is a $C_{0}(G_{1}^{0})\otimes_{*} C_{0}(G_{2}^{0})$-module where $\otimes_{*}$ denote the completion of algebraic tensor product $C_{0}(G_{1}^{0})\otimes C_{0}(G_{2}^{0})$ under the spatial $C^{*}$-norm. More details can be referred to in \cite{lance1995hilbert}. Since $C_{0}(G_{1}^{0})$ is commutative, it is nuclear, and hence the $C^*$-norm is unique. Using \cite[Theorem $6.4.17$]{murphy2014c} , $C_{0}(G_{1}^{0})\otimes_{*} C_{0}(G_{2}^{0})\cong C_{0}(G_{1}^{0}\times G_{2}^{0})$. Also we can easily see that the pre-Hilbert $C_{0}(G_{1}^{0}\times G_{2}^{0})$-module  $\mathcal{F}^{\sigma_{1}}_{0}(G_{1})\otimes \mathcal{F}^{\sigma_{2}}_{0}(G_{2})$ is dense in $\mathcal{F}^{\sigma_{1}}(G_{1})\otimes \mathcal{F}^{\sigma_{2}}(G_{2})$.

   Let $\{\mu_{1}^{u}\times\mu_{2}^{v}\}_{(u,v)\in G_{1}^{0}\times G_{2}^{0}}$ is the equivariant $r_{G^{0}\times G^{0}}$-system on the left $G_{1}\times G_{2}$-space $(G_{1}\times G_{2})/(H_{1}\times H_{2})\cong G_{1}/H_{1}\times G_{2}/H_{2}$, where $\mu_{1}^{u}$ and $\mu_{2}^{v}$ are equivariant $r_{G^{0}}$- systems of $G_{1}/H_{1}$ and $G_{2}/H_{2}$ respectively.
   
   For $(f_{1},f_{2})\in \mathcal{F}^{\sigma_{1}}_{0}(G_{1})\times \mathcal{F}^{\sigma_{2}}_{0}(G_{2})$, we define a function from $G_{1}\times G_{2}$ to $\mathcal{H}^{1}\otimes \mathcal{H}^{2}$ such that 
   $$ (f_{1},f_{2})(x,y)= f_{1}(x)\otimes f_{2}(y).$$
   We can easily see that $(f_{1},f_{2}) \in \mathcal{F}^{\sigma_{1}\times \sigma_{2}}(G_{1}\times G_{2})$.
    Now, define a module map $\Phi$ from $\mathcal{F}^{\sigma_{1}}_{0}(G_{1})\otimes \mathcal{F}^{\sigma_{2}}_{0}(G_{2}) $ to $\mathcal{F}^{\sigma_{1}\times \sigma_{2}}(G_{1}\times G_{2})$ as
    $$ \Phi((f_{1}\otimes f_{2}))(x,y)= f_{1}(x)\otimes f_{2}(y).$$

    \begin{lem}
        The module map $\Phi$ preserves the $C_{0}(G_{1}^{0}\times G_{2}^{0})$- valued inner product.
    \end{lem}
    \begin{proof}
    \begin{align*}
        & \langle\Phi((f_{1}\otimes f_{2})), \Phi((g_{1}\otimes g_{2}))\rangle(u,v)\\
        = & \iint_{\frac{G_{1}}{H_{1}}\times \frac{G_{2}}{H_{2}}}\langle\Phi((f_{1},f_{2}))(x,y),\Phi((g_{1},g_{2}))(x,y)\rangle d\mu_{1}^{u}(xH_{1})d\mu_{2}^{v}(yH_{2})\\
        = & \iint_{\frac{G_{1}}{H_{1}}\times \frac{G_{2}}{H_{2}}}\langle f_{1}(x)\otimes f_{2}(y),g_{1}(x)\otimes g_{2}(y)\rangle d\mu_{1}^{u}(xH_{1})d\mu_{2}^{v}(yH_{2})\\
        = & \iint_{\frac{G_{1}}{H_{1}}\times \frac{G_{2}}{H_{2}}}\langle f_{1}(x),g_{1}(x)\rangle \langle f_{2}(y), g_{2}(y)\rangle d\mu_{1}^{u}(xH_{1})d\mu_{2}^{v}(yH_{2})\\
        = & \int_{G_{1}/{H_{1}}} \langle f_{1}(x),g_{1}(x)\rangle d\mu_{1}^{u}(xH_{1}) \int_{G_{2}/{H_{2}}} \langle f_{2}(y), g_{2}(y)\rangle d\mu_{2}^{v}(yH_{2})\\
        = & \langle f_{1},g_{1}\rangle(u) \langle f_{2},g_{2}\rangle(v)= \langle f_{1}\otimes f_{2},g_{1}\otimes g_{2}\rangle (u,v).\qedhere
    \end{align*}        
    \end{proof}
    \begin{lem}
        Let $f_{i} \in C_{c}(G_{i})$ and $t_{i} \in C_{0}(G_{i}^{0},\mathcal{H}^{\sigma_{i}})$, $i=1,2$. Define $f\in C_{c}(G_{1}\times G_{2})$ by $f(x_{1},x_{2})=f_{1}(x_{1})f_{2}(x_{2})$. Then,
        $$ \mathcal{E}(f,t_{1}\otimes t_{2})(x_{1},x_{2})=\mathcal{E}(f_{1},t_{1})(x_{1})\otimes \mathcal{E}(f_{2},t_{2})(x_{2}).$$
    \end{lem}
    \begin{proof}
        Fix $(x_{1},x_{2}) \in G_{1}\times G_{2}$.  
        For all $t_{1}' \in C_{0}(G_{1}^{0},\mathcal{H}^{\sigma_{1}})$ and $t_{2}' \in C_{0}(G_{2}^{0},\mathcal{H}^{\sigma_{2}})$,
\begin{align*}
    & \Big\langle \mathcal{E}(f,t_{1}\otimes t_{2})(x_{1},x_{2}),t_{1}'(d(x_{1}))\otimes t_{2}'(d(x_{2}))\Big\rangle \\
     = & \iint\limits_{H_{1}\times H_{2}}f_{1}(x_{1}h_{1})f_{2}(x_{2}h_{2})\Big\langle\sigma_{1}(h_{1})t_{1}(d(h_{1})),t_{1}'(d(x_{1}))\Big\rangle\Big\langle\sigma_{2}(h_{2})t_{2}(d(h_{2})),t_{2}'(d(x_{2}))\Big\rangle  \\ & \hspace{100mm} d\lambda_{H_{1}}^{d(x_{1})}(h_{1})d\lambda_{H_{2}}^{d(x_{2})}(h_{2})\\
     = & \Big \langle\mathcal{E}(f_{1},t_{1})(x_{1}),t_{1}'(d(x_{1}))\Big\rangle \Big \langle\mathcal{E}(f_{2},t_{2})(x_{2}),t_{2}'(d(x_{2}))\Big\rangle\\
     = & \Big\langle \mathcal{E}(f_{1},t_{1})(x_{1})\otimes \mathcal{E}(f_{2},t_{2})(x_{2}),t_{1}'(d(x_{1}))\otimes t_{2}'(d(x_{2}))\Big\rangle. \qedhere 
\end{align*}  
\end{proof}
\begin{lem}
   The image of  $\mathcal{E}(C_{c}(G_{1}),C_{0}(G_{1}^{0}, \mathcal{H}^{\sigma_{1}})) \otimes \mathcal{E}(C_{c}(G_{2}),C_{0}(G_{2}^{0}, \mathcal{H}^{\sigma_{2}}))$ under $\Phi$ is total in $\mathcal{F}^{\sigma_{1}\times \sigma_{2}}(G_{1}\times G_{2})$.
\end{lem}
\begin{proof}
    By Lemma $3.6$, it is clear that  $\mathcal{E}(f,t)$ where $f\in C_{c}(G_{1}\times G_{2})$ and $t \in C_{0}(G_{1}^{0}\times G_{2}^{0}, \mathcal{H}^{\sigma_{1}}\otimes \mathcal{H}^{\sigma_{2}}) $ is a total subset of $\mathcal{F}^{\sigma_{1}\times \sigma_{2}}(G_{1}\times G_{2})$. Also note that $C_{0}(G_{1}^{0}, \mathcal{H}^{\sigma_{1}}) \otimes C_{0}(G_{2}^{0}, \mathcal{H}^{\sigma_{2}})$ is fibrewise dense in $\mathcal{H}^{\sigma_{1}}\otimes \mathcal{H}^{\sigma_{2}}$. Hence, we can say that  $\mathcal{E}(f,t_{1}\otimes t_{2})$, where $t_{i}\in C_{0}(G_{i}^{0}, \mathcal{H}^{i}),i=1,2,$ is a total subset in $\mathcal{F}^{\sigma_{1}\times \sigma_{2}}(G_{1}\times G_{2})$.
    So, it is enough to show that $\mathcal{E}(f,t_{1}\otimes t_{2})$, can be approximated by linear combination of $\Phi(\mathcal{E}(g_{1},t_{1}) \otimes \mathcal{E}(g_{2},t_{2})), g_{i} \in C_{c}(G_{i}),i=1,2$, in $\mathcal{F}^{\sigma_{1}\times \sigma_{2}}(G_{1}\times G_{2})$.

     Now, let $V_{i},i=1,2$, be two relatively compact subsets of $G_{i}, f\in C_{c}(G_{1}\times G_{2})$ with $supp(f)\subset V_{1}\times V_{2}$. Then, by Lemma $3.6$,
      \begin{equation}
          \|\mathcal{E}(g,t)\|\leq c \|g\|_{\infty}\|t\|
      \end{equation} for some $c>0$ and  for all $t \in C_{0}(G_{1}^{0}\times G_{2}^{0}, \mathcal{H}^{\sigma_{1}}\times \mathcal{H}^{\sigma_{2}})$ and $g\in C_{c}(G_{1}\times G_{2})$ with $ supp(g)\subset \Bar{V}_{1} \times \Bar{V}_{2}$.

      Given $\epsilon >0$, by Stone-Weierstrass theorem, there exist $f_{ij}\in C_{c}(G_{i}),i=1,2,j=1,...,n$ such that $supp(f_{ij})\subset V_{i}$ and
      $$ |f(x_{1},x_{2})-\sum\limits_{j=1}^{n}f_{1j}(x_1)f_{2j}(x_{2})|<\epsilon$$ for all $(x_{1},x_{2})\in G_{1}\times G_{2}$. Let $f_{j}(x_{1},x_{2})=f_{1j}(x_{1})f_{2j}(x_{2})$ and by Lemma $5.2$,
      \begin{align*}
          \sum\limits_{j=1}^{n}\Phi(\mathcal{E}(f_{1j},t_{1}) \otimes \mathcal{E}(f_{2j},t_{2}))(x_{1},x_{2})&=\sum\limits_{j=1}^{n}\mathcal{E}(f_{1j},t_{1})(x_{1})\otimes \mathcal{E}(f_{2j},t_{2})(x_{2})\\
          &= \sum\limits_{j=1}^{n} \mathcal{E}(f_{j},t_{1}\otimes t_{2})(x_{1},x_{2}).
      \end{align*}
      This implies that for $(x_{1},x_{2})\in (G_{1}\times G_{2})$, by calculation
      $$\mathcal{E}(f,t_{1}\otimes t_{2})-  \sum\limits_{j=1}^{n}\Phi(\mathcal{E}(f_{1j},t_{1}) \otimes \mathcal{E}(f_{2j},t_{2}))(x_{1},x_{2}) = \mathcal{E}\Big(f- \sum\limits_{j=1}^{n}f_{1j}f_{2j},t_{1}\otimes t_{2}\Big)(x_{1},x_{2}).$$
     Hence, using $(1)$, we have
      \begin{align*}
          & \|\mathcal{E}(f,t_{1}\otimes t_{2})-  \sum\limits_{j=1}^{n}\Phi(\mathcal{E}(f_{1j},t_{1}) \otimes \mathcal{E}(f_{2j},t_{2}))\| \\ = & \|\mathcal{E}\Big(f- \sum\limits_{j=1}^{n}f_{1j}f_{2j},t_{1}\otimes t_{2}\Big)\|
           \leq c \|t_{1}\|\|t_{2}\| \|f-\sum\limits_{j=1}^{n}f_{1j}f_{2j}\|_{\infty}
          \leq \epsilon ~c \|t_{1}\|\|t_{2}\|. \qedhere
      \end{align*}
\end{proof}

With this series of Lemmas, we can conclude that $\Phi$ is a surjective isometric $C_{0}(G_{1}^{0}\times G_{2}^{0})$-linear map from $\mathcal{F}^{\sigma_{1}}(G_{1})\otimes \mathcal{F}^{\sigma_{2}}(G_{2})$ to $\mathcal{F}^{\sigma_{1}\times \sigma_{2}}(G_{1}\times G_{2})$. Then by \cite[Theorem $3.5$]{lance1995hilbert}, Hilbert modules $\mathcal{F}^{\sigma_{1}}(G_{1})\otimes \mathcal{F}^{\sigma_{2}}(G_{2})$ and $\mathcal{F}^{\sigma_{1}\times \sigma_{2}}(G_{1}\times G_{2})$ are unitarily equivalent. Thus, by \cite[Theorem $4.2.4$]{bos2007groupoids}, $\Bar{\mathcal{F}}_{G_{1}}^{\sigma_{1}}\otimes \Bar{\mathcal{F}}_{G_{2}}^{\sigma_{2}}$ and $\Bar{\mathcal{F}}_{G_{1}\times G_{2}}^{\sigma_{1}\times \sigma_{2}}$ are isometrically isometric Hilbert bundles.

Now, we prove Mackey's tensor product theorem on locally compact groupoids. 
\begin{thm}
    Let $G_{1}$ and $G_{2}$ be two locally compact groupoids with closed wide subgroupoids $H_{1}$ and $H_{2}$. Then for any representations $\sigma_{1}$ of $H_{1}$ and $\sigma_{2}$ of $H_{2}$, the representations $\left(\operatorname{ind}_{H_{1}}^{G_{1}}\sigma_{1},\mu_{1}\right) \times \left(\operatorname{ind}_{H_{2}}^{G_{2}}\sigma_{2},\mu_{2} \right)$ and $\left(\operatorname{ind}_{H_{1}\times H_{2}}^{G_{1}\times G_{2}}(\sigma_{1}\times \sigma_{2}),\mu_{1}\times \mu_{2}\right)$ are equivalent.
\end{thm}
\begin{proof}
    From previous lemmas and results, it is enough to show that the morphism induced by $\Phi$, again denoted as $\Phi$, from $\Bar{\mathcal{F}}_{G_{1}}^{\sigma_{1}}\otimes \Bar{\mathcal{F}}_{G_{2}}^{\sigma_{2}}$ to $\Bar{\mathcal{F}}^{\sigma_{1}\times \sigma_{2}}_{G_{1}\times G_{2}}$ intertwines with $\left(\operatorname{ind}_{H_{1}}^{G_{1}}\sigma_{1},\mu_{1}\right) \times \left(\operatorname{ind}_{H_{2}}^{G_{2}}\sigma_{2},\mu_{2} \right)$ and $\left(\operatorname{ind}_{H_{1}\times H_{2}}^{G_{1}\times G_{2}}(\sigma_{1}\times \sigma_{2}),\mu_{1}\times \mu_{2}\right)$.

    Let $(x_{1},x_{2}), (y_{1},y_{2})\in G_{1}\times G_{2}$, $f_{1}\otimes f_{2}\in \mathcal{F}_{0}^{\sigma_{1}}\otimes \mathcal{F}_{0}^{\sigma_{2}}$, then
    \begin{align*}
       & \Big(\operatorname{ind}_{H_{1}\times H_{2}}^{G_{1}\times G_{2}}(\sigma_{1}\times \sigma_{2})(x_{1},x_{2})\Phi(f_{1}\otimes f_{2})\Big)(y_{1},y_{2}) \\ = &\Phi(f_{1}\otimes f_{2})(x_{1}^{-1}y_{1},x_{2}^{-1}y_{2}) = f_{1}(x_{1}^{-1}y_{1})\otimes f_{2}(x_{2}^{-1}y_{2})\\
       = & (\operatorname{ind}_{H_{1}}^{G_{1}}\sigma_{1}(x_{1})f_{1})(y_{1})\otimes (\operatorname{ind}_{H_{2}}^{G_{2}}\sigma_{2}(x_{2})f_{2})(y_{2})\\
       = & \Phi\Big((\operatorname{ind}_{H_{1}}^{G_{1}}\sigma_{1}(x_{1})f_{1})\otimes (\operatorname{ind}_{H_{2}}^{G_{2}}\sigma_{2}(x_{2})f_{2})\Big)(y_{1},y_{2})\\
       = & \Phi\Big((\operatorname{ind}_{H_{1}}^{G_{1}}\sigma_{1} \times \operatorname{ind}_{H_{2}}^{G_{2}}\sigma_{2})(x_{1},x_{2})(f_{1}\otimes f_{2})\Big)(y_{1},y_{2}).\qedhere
    \end{align*}
\end{proof}
\begin{section}{Frobenius Reciprocity}
The Frobenius reciprocity theorem relates the intertwining operators of a
representation with an induced representation to those of its restriction.
In this section we prove a version of this theorem for compact transitive
groupoids. The idea comes from the classical theorem for compact groups, however, the classical theorem does not apply
directly. We first identify the fibres of the induced bundle with induced
spaces of the isotropy groups, using the Haar system of $G$. We then show
that the intertwining operators are continuous bundle morphisms. This uses
the transitivity of $H$ and the compactness of $G$.

Frobenius reciprocity for groupoids without topology was studied by
Barbar\'an S\'anchez and El Kaoutit \cite{BarbaranElKaoutit}, using induction and co-induction functors. Here we work with continuous representations on continuous Hilbert bundles. Unlike the algebraic case, our theorem applies to groupoids with a non-discrete topology, and the intertwining operators it compares are continuous bundle morphisms.

Let $\Sigma_{0}$ be the space of closed subgroups of $G$ viewed as a subset of the collection of closed subsets of $G$ under the Fell topology. One can define a continuous function $p:\Sigma_{0}\to G^{0}$, given by $p(B)=u$ if $B\subset G^{u}_{u}$. Define $G*\Sigma_{0}=\{(g,B)\in G\times \Sigma_{0}:d(g)=p(B)\}$. Then $G$ acts continuously on $\Sigma_{0}:$ If $(g,B)\in G*\Sigma_{0}$,~ $g\cdot B= \{gtg^{-1}:t\in B\}$. For more details refer \cite[Section $3.4$]{williams2019tool}.

Suppose $G$ is a compact transitive groupoid  and $H$ a closed wide subgroupoid with a normalised Haar system $\{\lambda_{H}^{u}\}_{u\in G^{0}}$, the existence of a full equivariant system of measure $\{\mu^{u}\}_{u\in G^{0}}$ on $G/H$ is guaranteed by \cite[Proposition $2.5$]{MR3478528}.  
  Since $G$ is compact transitive groupoid, by \cite[Proposition $3.18$]{edeko2022uniform} and \cite[Lemma $1.3$]{renault1991ideal}, a family of Haar measures $\{\beta^{u}_{u}\}_{u\in G^{0}}$ form a Haar system on the isotropy subgroupoid $G'$. 
 Given $\{\mu^{u}\}_{u\in G^{0}}$, we can define another Haar system $\{\lambda'^{u}\}_{u\in G^{0}}$ on $G$ as 
   \begin{equation}
     \int gd\lambda'^{u} = \int Pg d\mu^{u}, ~~g\in C_{c}(G).  
   \end{equation}

  By \cite[Proposition $6.15$]{williams2019tool}, one can write 
  \begin{equation}
      \lambda'^{u}(f)= \int_{R}\int_{G^{u}_{v}}f(t)d\beta^{u}_{v}(t)dm^{u}(v)
  \end{equation}  where $R$ is the range of $(r,s):G\to G^{0}\times G^{0}, (r,s)(x)=(r(x),s(x))$  which forms a topological groupoid under quotient topology and $\{m^{u}\}_{u\in G^{0}}$ is the Haar system for $R$.

  Suppose $f \in \mathcal{F}_{0}^{\sigma}(G)$, then the function $g$, $g(x)=\langle f(x),f(x)\rangle_{\mathcal{H}_{d(x)}}$ is a continuous function. Since $\lambda_{H}$ is a normalized Haar system of $H$, we can easily verify that $Pg(xH)= \int_{H}g(xh)d\lambda_{H}(h)= g(x)$. Hence, by $(2)$,
   $$\langle f,f\rangle(u)= \int_{G/H}Pg(xH)d\mu^{u}(xH)= \int_{G} \langle f(x),f(x)\rangle_{\mathcal{H}_{d(x)}} d\lambda'^{u}(x).$$
   The following is a groupoid version of the Frobenius Reciprocity theorem.

    \begin{thm}
        Let $G$ be a transitive compact groupoid, $H$ a closed transitive wide subgroupoid with a normalized Haar system  $\{\lambda_{H}^{u}\}_{u\in G^{0}}$ and $\{\mu^{u}\}_{u\in G^{0}}$ be a full equivariant measure system on $G/H$. If $\pi$ is an irreducible representation of $G$, and $\sigma$ an irreducible representation of $H$, then,
        $$Mor\left(\pi,\left(\operatorname{ind}_{H}^{G}(\sigma),\mu \right)\right)\cong Mor(\pi_{|_{H}},\sigma).$$
    \end{thm}
    \begin{proof}
         Note that, $ \beta^{u}_{v}= x\beta^{v}_{v}= \delta(x, G^{v}_{v}) \beta^{u}_{u}x,~ x\in G^{u}_{v}$, where $\delta$ is a positive continuous function on $G*\Sigma_{0}$ as defined in  \cite[Lemma $3.25$]{williams2019tool}. Since each isotropy subgroup is compact, one can easily see that $\delta$ is constant on $G^{u}_{v}$ and $\beta'^{u}_{v}= \beta^{u}_{u}x$ is independent of $x\in G^{u}_{v}$. Thus, we can see that $\beta^{u}_{v}= \delta'(v)\beta'^{u}_{v}$, where $\delta'$ is a positive continuous function on $G^{u}$ defined as $\delta'(v)=\delta(x, G^{v}_{v})$ for some $x\in G^{u}_{v}$. Suppose $f \in \mathcal{F}_{0}^{\sigma}(G)$, $u\in G^{0}$, we can see that $\int_{G^{u}_{u}} \langle f(x),f(x)\rangle_{\mathcal{H}^{\sigma}_{u}}d\beta_{u}^{u}(x) < \infty$ and  \begin{align*}
            \int \|f(x)\|^{2}_{\mathcal{H}^{\sigma}_{v}}d\beta'^{u}_{v}(x)=\int \|f(xh)\|^{2}_{\mathcal{H}^{\sigma}_{v}}d\beta^{u}_{u}(x)=\int \|f(x)\|^{2}_{\mathcal{H}^{\sigma}_{u}}d\beta^{u}_{u}(x),~~h\in H^{u}_{v}.
        \end{align*} Thus by $(3)$, and above discussions, for each $u\in G^{0}$, we get 
        $$\int_{G} \langle f(x),f(x)\rangle_{\mathcal{H}^{\sigma}_{u}}d\beta_{u}^{u}(x)= K^{-1}_{u} \int_{G} \langle f(x),f(x)\rangle_{\mathcal{H}^{\sigma}_{d(x)}}d\lambda'^{u}(x),$$
where $K_{u}= \int_{G^{0}}\delta'(v)dm^{u}(v)$ and the family of  measures $\{\lambda'^{u}\}_{u\in G^{0}}$ is the Haar system defined in $(2)$.
        By \cite[Lemma $4.10$]{edeko2022uniform},  $\pi_{|_{G^{u}_{u}}}$ is irreducible for every $u \in G^{0}$. Let $\mathcal{E}^{\Bar{\pi}}_{u,u}$ denote the linear span of matrix elements of $\Bar{\pi}$ on $G^{u}_{u}$. Now, suppose $T\in Mor(\pi,\left(\operatorname{ind}_{H}^{G}(\sigma),\mu \right))$, then by Peter-Weyl theorem, for $v\in \mathcal{H}^{\pi}_{u}$, $T_{u}v_{|_{G^{u}_{u}}}$ is contained in the $d_{\sigma}^{u}$ copies of  $\mathcal{E}^{\Bar{\pi}}_{u,u}$, which are continuous on $G^{u}_{u}$, where $d_{\sigma}^{u}$ is the dimension of $\mathcal{H}^{\sigma}_{u}$. Thus, it makes sense to evaluate $T_{u}v$ pointwise at $G^{u}_{u}$ for every $u\in G^{0}$. Also, by transitivity of $H$,  $T_{u}v \in C(G^{u},\mathcal{H}^{\sigma})$ for every $u\in G^{0}$.

        Let $E_{u}:C(G^{u},\mathcal{H}^{\sigma}) \to \mathcal{H}^{\sigma}_{u}$ be the evaluation map $E_{u}f=f(u)$. Then, let $(ET)_{u}=E_{u}T_{u}$ and for $u,w \in G^{0}, v\in \mathcal{H}^{\pi}_{u}, h\in H\cap G^{w}_{u}$
        \begin{align*}
            \sigma(h)(ET)_{u}v &= \sigma(h)[T_{u}v(u)]=T_{u}v(uh^{-1})\\ = & T_{u}v(h^{-1}w) =[(\operatorname{ind}_{H}^{G}\sigma)(h)T_{u}v](w) \\ = & [T_{w}\pi(h)v](w)= (ET)_{w}\pi(h)v.
        \end{align*} Fix $u_{0} \in G^{0}$. Let $\eta \in C_{0}(H^{0},\mathcal{H}^{\sigma})$ and $\xi \in C_{0}(G^{0},\mathcal{H}^{\pi})$, then note that $$\Big\langle[(ET)\xi](u),\eta (u)\Big\rangle= \Big\langle \sigma(h)[(ET)\xi](u),\sigma(h) \eta(u)\Big\rangle = \Big\langle (ET)_{u_{0}}\pi(h)\xi(u), \sigma(h) \eta(u)\Big\rangle$$ and $\|(ET)\xi(u)\|=\|(ET)_{u_{0}}\pi(h)\xi(u)\|$ for some $h \in H\cap G^{u_{0}}_{u}$. Using transitivity of $H$ and compactness argument, it is easy to see that $[(ET)\xi](u)=[T_{u}\xi^{u}](u)$ is continuous as a section.
        So, $ET \in Mor(\pi_{|_{H}},\sigma)$. 
        
        Now, if $ET=0$, then for any $u,w \in G^{0}, x\in G^{w}_{u},v\in \mathcal{H}^{\pi}_{w}$,
        $$ 0=[T_{u}\pi(x^{-1})v](u)=[\operatorname{ind}_{H}^{G}\sigma(x^{-1})T_{w}v](u)=T_{w}v(x),$$
        so $T=0$. Thus, $T\to ET$ is injective. Next, we prove that the map is surjective also.  

        If $S\in Mor(\pi_{|_{H}},\sigma) $, define $T: \mathcal{H}^{\pi}\to \Bar{\mathcal{F}}^{\sigma}$ as $[T_{u}v](x)=S_{d(x)}[\pi(x^{-1})v]$ for every $x\in G^{u},v \in \mathcal{H}^{\pi}_{u},u\in G^{0}$.
        \begin{align*}
          [T_{u}v](xh)= S_{d(h)}[\pi(h^{-1})\pi(x^{-1})v]=\sigma(h^{-1})S_{r(h)}[\pi(x^{-1})v] =  \sigma(h^{-1})[T_{u}v](x).
        \end{align*}
        Note that, $T_{u}v \in C(G^{u},\mathcal{H}^{\sigma})$. Also $T\xi \in \mathcal{F}^{\sigma}(G)$ for $\xi \in C_{0}(G^{0},\mathcal{H}^{\pi})$. 
        
         For $y \in G, v\in \mathcal{H}^{\pi}_{d(y)}$,
        \begin{align*}
            [\operatorname{ind}_{H}^{G}(\sigma)(y)T_{d(y)}v](x)=[T_{d(y)}v](y^{-1}x)=S_{d(x)}[\pi(x^{-1})\pi(y)v]=T_{r(y)}[\pi(y)v](x).
        \end{align*}
        Hence, $T\in Mor(\pi,\left(\operatorname{ind}_{H}^{G}(\sigma),\mu \right)) $ and for $u\in G^{0}$, $(ET)_{u}v= [T_{u}v](u)= S_{u}v$. So, $S=ET$. \qedhere
        
    \end{proof}
    A continuous unitary representation is internally irreducible if the representation is irreducible when restricted to each isotropy subgroup. More on this can be found in  \cite[Definition $13$]{bos2011continuous}. 
   \begin{corl}
Let $G$ be a compact groupoid with finite unit space, and let $H$ be a closed
wide subgroupoid with a normalized Haar system such that $H$ and $G$ have the
same orbits. Then the conclusion of Theorem 6.1 holds for every
internally irreducible representation $\pi$ of $G$ and $\sigma$ of $H$.
\end{corl}
\begin{proof}
We follow the proof of Theorem 6.1. Since $H$ and $G$ have the same orbits,
for $u,w$ in the same orbit there is an element of $H$ from $w$ to $u$. This
is the only property of $H$ used to identify the fibres of the induced bundle
with induced spaces of the isotropy groups. The decomposition (3) of the Haar
system holds since $G^{0}$ is finite, and $\pi|_{G^u_u}$ is irreducible since
$\pi$ is internally irreducible. Since $G^0$ is finite and discrete, every
bundle map over $G^0$ is continuous. Moreover, $G^u$ is the disjoint union of
the finitely many open and closed sets $G^u_w$, with $w$ in the orbit of $u$.
For each such $w$, fix $h_w\in H$ with $r(h_w)=w$ and $d(h_w)=u$. Then
$T_uv(x)=\sigma(h_w)\,T_uv(xh_w)$ for $x\in G^u_w$, so $T_uv$ is continuous on
$G^u$, since it is continuous on $G^u_u$. Hence the result holds.
\end{proof}
\end{section}
\begin{section}{Comparison with Rieffel induction}

In this section, we show that our construction agrees with the induced
representation of \cite{MR2465655}. In \cite{MR2465655}, induction is done by
Rieffel's method: a representation of the subgroupoid $C^{*}$-algebra
$C^{*}(H)$ on a Hilbert space is induced to a representation of the groupoid
$C^{*}$-algebra $C^{*}(G)$, and a representation of the groupoid is then
recovered using Renault's disintegration theorem. Let $\nu$ be an invariant
probability measure on $G^{0}$, and let $(\mathcal{H}^{\sigma},\sigma)$ be a
continuous unitary representation of $H$. We show that the representation of
$C^{*}(G)$ induced from the integrated form of $\sigma$ on the completion of
$C_{0}(G^{0},\mathcal{H}^{\sigma})$ with respect to $\nu$ is unitarily
equivalent to the integrated form of
$\big(\bar{\mathcal{F}}^{\sigma},\operatorname{ind}^{G}_{H}(\sigma,\mu)\big)$.

Let $\mu$ be a full equivariant $r_{G^0}$-system on $G/H$, and let $\lambda'$
be a Haar system on $G$ satisfying (2). Let $(G,\lambda',\nu)$ be a pmp
groupoid and $(H,\lambda_{H},\nu)$ a wide pmp subgroupoid, that is, $\nu$ is a
probability measure on $G^0$ which is invariant for both $\lambda'$ and
$\lambda_H$. For instance, these conditions hold in Example 2, with $\nu$ the
Lebesgue measure on $\mathbb{T}$, and in Example 3 with
$\mu^{u}=\delta_{u}\times m$ and $\nu=m$, where $m$ is a probability measure
with full support. Note that the groupoid in Example 3 is not \'etale. Given a
continuous unitary representation $(\mathcal{H}^{\sigma},\sigma)$ of $H$, we
consider the measurable Hilbert bundle $(\mathcal{H}^{\sigma},\nu)$, whose
space of $L^{2}$-sections is the completion of
$C_{0}(G^{0},\mathcal{H}^{\sigma})$ with respect to $\nu$. Similarly,
$L^2(G^0,\bar{\mathcal{F}}^\sigma_G)$ denotes the space of $L^2$-sections of
$\bar{\mathcal{F}}^\sigma_G$ with respect to $\nu$.
\begin{thm}
Let $(G,\lambda',\nu)$ be a pmp groupoid and $(H,\lambda_{H},\nu)$ a wide pmp
subgroupoid. Then the representation of $C^*(G)$ induced from the integrated
form of $\sigma$ by Rieffel's method is unitarily equivalent to the integrated
form of $\operatorname{ind}_H^G(\sigma,\mu)$ with respect to $\nu$.
\end{thm}

\begin{proof}
Let $\phi,\psi\in C_{c}(G)$ and $t,s\in C_{0}(G^{0},\mathcal{H}^{\sigma})$.
Following \cite{MR2465655}, consider the pre-Hilbert space
$C_{c}(G)\odot C_{0}(G^{0},\mathcal{H}^{\sigma})$ with the inner product
$$\langle \phi \otimes t,\psi \otimes s\rangle
=\langle \sigma (\langle\psi,\phi \rangle_{*})t,s\rangle,
\qquad
\langle\psi,\phi \rangle_{*}(h)=\int_{G} \overline{\psi(y)}\,\phi(yh)\,
d\lambda'_{r(h)}(y).$$
By Lemma 3.1, choose $b\in C_{c}(G)$ with $Pb=1$ on
$q_H(\operatorname{supp}\psi)$. For $h,h'\in H$, write $\Theta(h)=\langle \sigma(h)t(d(h)),s(r(h))\rangle$ and,
when $r(h)=r(h')$,
$\Theta(h,h')=\langle \sigma(h)t(d(h)),\sigma(h')s(d(h'))\rangle$. Then
\begin{align*}
&\int_{G^0}\int_{H}\langle\psi,\phi \rangle_{*}(h)\,\Theta(h)\,
d\lambda_{H}^{u}(h)\,d\nu(u)\\
&=\int_{G^0}\int_{H}\int_{G} \overline{\psi(y)}\,\phi(yh)\,\Theta(h)\,
d\lambda'_{r(h)}(y)\,d\lambda_{H}^{u}(h)\,d\nu(u)\\
&=\int_{G^0}\iint_{H}\int_{G} \overline{\psi(y)}\,\phi(yh)\,b(yh')\,\Theta(h)\,
d\lambda'_{r(h)}(y)\,d\lambda_{H}^{u}(h)\,d\lambda_{H}^{d(y)}(h')\,d\nu(u)\\
&=\int_{G^0}\iint_{H}\int_{G} \overline{\psi(yh'^{-1})}\,\phi(yh'^{-1}h)\,b(y)\,\Theta(h)\,
d\lambda'_{d(h')}(y)\,d\lambda_{H}^{u}(h')\,d\lambda_{H}^{u}(h)\,d\nu(u)\\
&=\int_{G^0}\int_{G}\iint_{H} \overline{\psi(yh')}\,\phi(yh'h)\,b(y)\,\Theta(h)\,
d\lambda_{H}^{d(h')}(h)\,d\lambda_{H}^{d(y)}(h')\,d\lambda'_{u}(y)\,d\nu(u)\\
&=\int_{G^0}\int_{G}\iint_{H} \overline{\psi(yh')}\,\phi(yh)\,b(y)\,\Theta(h,h')\,
d\lambda_{H}^{r(h')}(h)\,d\lambda_{H}^{d(y)}(h')\,d\lambda'_{u}(y)\,d\nu(u)\\
&=\int_{G^0}\int_{G}\langle \mathcal{E}(\phi,t)(y), \mathcal{E}(\psi,s)(y)\rangle\,
b(y)\, d\lambda'_{u}(y)\,d\nu(u)\\
&=\int_{G^0}\int_{G}\langle \mathcal{E}(\phi,t)(y), \mathcal{E}(\psi,s)(y)\rangle\,
b(y)\, d\lambda'^{u}(y)\,d\nu(u).
\end{align*}
where the last equality uses the invariance of $\nu$. Since
$y\mapsto\langle \mathcal{E}(\phi,t)(y),\mathcal{E}(\psi,s)(y)\rangle$ is
constant on cosets and $Pb=1$ on its support, (2) gives
$$\int_{G}\langle \mathcal{E}(\phi,t)(y), \mathcal{E}(\psi,s)(y)\rangle\,
b(y)\, d\lambda'^{u}(y)
=\langle \mathcal{E}(\phi,t),\mathcal{E}(\psi,s)\rangle(u).$$
Hence the map
$U\colon C_{c}(G)\odot C_{0}(G^{0},\mathcal{H}^{\sigma})\to
L^{2}(G^{0},\bar{\mathcal{F}}^{\sigma}_G)$,
$U(\phi \otimes t)= \mathcal{E}(\phi,t)$, preserves inner products. Its range
is dense by Lemma 3.6, since $a\cdot\mathcal{E}(\phi,t)=\mathcal{E}((a\circ
r)\phi,t)$ for $a\in C_0(G^0)$. So $U$ extends to a unitary map.

We show that $U$ intertwines the two representations. Let $f\in C_c(G)$, and
let $F$ be in $\mathcal{F}^\sigma_0(G)$, viewed as an element of
$L^2(G^0,\bar{\mathcal{F}}^\sigma_G)$. Choosing $b$ as above with $Pb=1$ on
$q_H(\operatorname{supp}F)$, we get
\begin{align*}
&\big\langle U\big((f*\phi)\otimes t\big),F\big\rangle\\
=& \langle \mathcal{E}(f*\phi,t),F\rangle\\
=& \int_{G^{0}}\int_{G/H}\langle \mathcal{E}(f*\phi,t)(x),F(x)\rangle\,
d\mu^{u}(xH)\,d\nu(u)\\
=& \int_{G^{0}}\int_{G}\langle \mathcal{E}(f*\phi,t)(x),F(x)\rangle\, b(x)\,
d\lambda'^{u}(x)\,d\nu(u)\\
=& \int_{G^{0}}\int_{G}\int_{H} (f*\phi)(xh)\,\langle\sigma(h)t(d(h)),F(x)\rangle\,
b(x)\,d\lambda_{H}^{d(x)}(h)\, d\lambda'^{u}(x)\,d\nu(u)\\
=& \int_{G^{0}}\iint_{G} f(y)\, \langle\mathcal{E}(\phi,t)(y^{-1}x),F(x)\rangle\,
b(x)\,d\lambda'^{u}(y)\, d\lambda'^{u}(x)\,d\nu(u)\\
=& \int_{G^{0}}\int_{G} f(y)\int_{G/H} \langle\mathcal{E}(\phi,t)(y^{-1}x),F(x)\rangle\,
d\mu^{r(y)}(xH)\, d\lambda'^{u}(y)\,d\nu(u)\\
=& \int_{G^{0}}\int_{G} f(y)\,
\big\langle \operatorname{ind}^{G}_{H}(\sigma,\mu)(y)\,\mathcal{E}(\phi,t)^{d(y)},
F^{r(y)}\big\rangle\, d\lambda'^{u}(y)\,d\nu(u)\\
=& \big\langle \operatorname{ind}^{G}_{H}(\sigma,\mu)(f)\,U(\phi \otimes t),F\big\rangle.
\end{align*}
\end{proof}
\begin{rem}
     \begin{enumerate}[(i)]
     \item The theorem is a groupoid analogue of Rieffel's result that his induced representations of groups agree with those of Mackey\cite[Theorem 5.12]{rieffel1974induced}.
     \item For groups, every unitary representation is continuous, so the  question of continuity does
not arise. For groupoids, the induced representation of \cite{MR2465655} is a
representation of $C^*(G)$, and Renault's disintegration theorem only gives a
Borel representation of $G$, determined up to $\nu$-null sets. The theorem
shows that, for an invariant measure $\nu$, Rieffel's method preserves
continuity: the representation induced from a continuous representation of
$H$ has a version on the measurable bundle $(\bar{\mathcal{F}}^{\sigma}_G,\nu)$
whose coefficients are continuous on the dense set
$C_0(G^0,\bar{\mathcal{F}}^{\sigma}_G)$ of sections. 
     \end{enumerate}
\end{rem}
\end{section}
\section*{Acknowledgement}
K. N. Sridharan is supported by the NBHM doctoral fellowship with Ref. number: 0203/13(45)/2021-R\&D-II/13173.  
\section*{Data Availability}
Data sharing does not apply to this article as no datasets were generated or analyzed during the current study.
\section*{competing interests}
The authors declare that they have no competing interests.
\section*{Author Contributions}
Both authors contributed equally to
this work, and read and approved the final manuscript.
 \bibliographystyle{abbrv}
\bibliography{reference}

@book{renault2006groupoid,
  title={A groupoid approach to C*-algebras},
  author={Renault, Jean},
  volume={793},
  year={2006},
  publisher={Springer}
}

@book{paterson2012groupoids,
  title={Groupoids, inverse semigroups, and their operator algebras},
  author={Paterson, Alan},
  volume={170},
  year={2012},
  publisher={Springer Science \& Business Media}
}

@phdthesis{bos2007groupoids,
  title={Groupoids in geometric quantization},
  author={Bos, Rogier David},
  year={2007},
  school={Nijmegen:[Sn]}
}

@article{renault1991ideal,
  title={The ideal structure of groupoid crossed product C*-algebras},
  author={Renault, Jean and SKANDALIS, GEORGES},
  journal={Journal of Operator Theory},
  pages={3--36},
  year={1991},
  publisher={JSTOR}
}

@book{folland2016course,
  title={A course in abstract harmonic analysis},
  author={Folland, Gerald B},
  year={2016},
  publisher={CRC press}
}

@article{bos2011continuous,
  title={Continuous representations of groupoids.(English summary)},
  author={Bos, Rogier},
  journal={Houston J. Math},
  volume={37},
  number={3},
  pages={807--844},
  year={2011},
  publisher={Citeseer}
}

@book{williams2019tool,
  title={Tool Kit for Groupoid $C^{*}$-Algebras},
  author={Williams, Dana P},
  volume={241},
  year={2019},
  publisher={American Mathematical Soc.}
}

@book{lance1995hilbert,
  title={Hilbert C*-modules: a toolkit for operator algebraists},
  author={Lance, E Christopher},
  volume={210},
  year={1995},
  publisher={Cambridge University Press}
}

@article{BSMF_1963__91__227_0,
     author = {Dixmier, Jacques and Douady, Adrien},
     title = {Champs continus d{\textquoteright}espaces hilbertiens et de $C^\ast $-alg\`ebres},
     journal = {Bulletin de la Soci\'et\'e Math\'ematique de France},
     pages = {227--284},
     publisher = {Soci\'et\'e math\'ematique de France},
     volume = {91},
     year = {1963},
     doi = {10.24033/bsmf.1596},
     mrnumber = {29 #485},
     zbl = {0127.33102},
     language = {fr},
     url = {http://www.numdam.org/articles/10.24033/bsmf.1596/}
}

@book{murphy2014c,
  title={C*-algebras and operator theory},
  author={Murphy, Gerald J},
  year={2014},
  publisher={Academic press}
}

@book{fell1988representations,
  title={Representations of *-algebras, Locally Compact Groups, and Banach *-algebraic Bundles: Basic representation theory of groups and algebras},
  author={Fell, James Michael Gardner and Doran, Robert S},
  year={1988},
  publisher={Academic press}
}

@article {MR3478528,
    AUTHOR = {Williams, Dana P.},
     TITLE = {Haar systems on equivalent groupoids},
   JOURNAL = {Proc. Amer. Math. Soc. Ser. B},
  FJOURNAL = {Proceedings of the American Mathematical Society. Series B},
    VOLUME = {3},
      YEAR = {2016},
     PAGES = {1--8},
      ISSN = {2330-1511},
   MRCLASS = {22A22 (28C15 46L05 46L55)},
  MRNUMBER = {3478528},
MRREVIEWER = {Lisa\ Orloff\ Clark},
       DOI = {10.1090/bproc/22},
       URL = {https://doi.org/10.1090/bproc/22},
}

@article{edeko2022uniform,
  title={Uniform enveloping semigroupoids for groupoid actions},
  author={Edeko, Nikolai and Kreidler, Henrik},
  journal={Journal d'Analyse Math{\'e}matique},
  volume={148},
  number={2},
  pages={739--796},
  year={2022},
  publisher={Springer}
}

@article{mackey1952induced,
  title={Induced representations of locally compact groups I},
  author={Mackey, George W},
  journal={Annals of mathematics},
  volume={55},
  number={1},
  pages={101--139},
  year={1952},
  publisher={JSTOR}
}

@article{mackey1953induced,
  title={Induced representations of locally compact groups II. The Frobenius reciprocity theorem},
  author={Mackey, George W},
  journal={Annals of Mathematics},
  volume={58},
  number={2},
  pages={193--221},
  year={1953},
  publisher={JSTOR}
}

@article{hermann1992induced,
  title={Induced representations of hypergroups},
  author={Hermann, Peter},
  journal={Mathematische Zeitschrift},
  volume={211},
  number={1},
  pages={687--699},
  year={1992},
  publisher={Springer}
}

@book{kaniuth2013induced,
  title={Induced representations of locally compact groups},
  author={Kaniuth, Eberhard and Taylor, Keith F},
  number={197},
  year={2013},
  publisher={Cambridge university press}
}

@book{gierz2006bundles,
  title={Bundles of topological vector spaces and their duality},
  author={Gierz, Gerhard},
  volume={955},
  year={2006},
  publisher={Springer}
}

@article{rieffel1974induced,
  title={Induced representations of $C^*$-algebras},
  author={Rieffel, Marc A},
  journal={Advances in Mathematics},
  volume={13},
  number={2},
  pages={176--257},
  year={1974},
  publisher={Academic Press}
}

@article{mautner1951generalization,
  title={A generalization of the Frobenius reciprocity theorem},
  author={Mautner, FI},
  journal={Proceedings of the National Academy of Sciences},
  volume={37},
  number={7},
  pages={431--435},
  year={1951},
  publisher={National Acad Sciences}
}

@article{Moore1962OnTF,
  title={On the Frobenius reciprocity theorem for locally compact groups.},
  author={Calvin C. Moore},
  journal={Pacific Journal of Mathematics},
  year={1962},
  volume={12},
  pages={359-365},
  url={https://api.semanticscholar.org/CorpusID:120129201}
}

@article{mackey1949imprimitivity,
  title={Imprimitivity for representations of locally compact groups I},
  author={Mackey, George W},
  journal={Proceedings of the National Academy of Sciences},
  volume={35},
  number={9},
  pages={537--545},
  year={1949},
  publisher={National Acad Sciences}
}

@article{mackey1951induced,
  title={On induced representations of groups},
  author={Mackey, George W},
  journal={American Journal of Mathematics},
  volume={73},
  number={3},
  pages={576--592},
  year={1951},
  publisher={JSTOR}
}

@article{blattner1961induced,
  title={On induced representations},
  author={Blattner, Robert J},
  journal={American Journal of Mathematics},
  volume={83},
  number={1},
  pages={79--98},
  year={1961},
  publisher={JSTOR}
}

@article {MR2465655,
    AUTHOR = {Ionescu, Marius and Williams, Dana P.},
     TITLE = {Irreducible representations of groupoid {$C^*$}-algebras},
   JOURNAL = {Proc. Amer. Math. Soc.},
  FJOURNAL = {Proceedings of the American Mathematical Society},
    VOLUME = {137},
      YEAR = {2009},
    NUMBER = {4},
     PAGES = {1323--1332},
      ISSN = {0002-9939,1088-6826},
   MRCLASS = {46L55 (22A22 46L05)},
  MRNUMBER = {2465655},
MRREVIEWER = {Qing\ Xiang\ Xu},
       DOI = {10.1090/S0002-9939-08-09782-7},
       URL = {https://doi.org/10.1090/S0002-9939-08-09782-7},
}

@book {MR2547343,
    AUTHOR = {Muhly, Paul S. and Williams, Dana P.},
     TITLE = {Renault's equivalence theorem for groupoid crossed products},
    SERIES = {New York Journal of Mathematics. NYJM Monographs},
    VOLUME = {3},
 PUBLISHER = {State University of New York, University at Albany, Albany,
              NY},
      YEAR = {2008},
     PAGES = {87},
   MRCLASS = {46L55 (22D25 46L05)},
  MRNUMBER = {2547343},
MRREVIEWER = {Claire\ Anantharaman-Delaroche},
}

@article{BarbaranElKaoutit,
  author  = {Barbar{\'a}n S{\'a}nchez, Juan Jes{\'u}s and El Kaoutit, Laiachi},
  title   = {Linear representations and {F}robenius morphisms of groupoids},
  journal = {SIGMA Symmetry Integrability Geom. Methods Appl.},
  volume  = {15},
  pages   = {Paper No. 019, 33},
  year    = {2019},
  doi     = {10.3842/SIGMA.2019.019}
}

@book{Bogachev2007,
  author    = {Bogachev, Vladimir I.},
  title     = {Measure Theory},
  volume    = {2},
  year      = {2007},
  publisher = {Springer-Verlag},
  address   = {Berlin, Heidelberg},
  pages     = {xiv + 575},
  isbn      = {978-3-540-34513-8},
  doi       = {10.1007/978-3-540-34514-5}
}

@article{Pysiak2004,
  author  = {Pysiak, Leszek},
  title   = {Groupoids, their representations and imprimitivity systems},
  journal = {Demonstratio Mathematica},
  volume  = {37},
  number  = {3},
  pages   = {661--670},
  year    = {2004},
  doi     = {10.1515/dema-2004-0316}
}

@article{Pysiak2011,
  author  = {Pysiak, Leszek},
  title   = {Imprimitivity theorem for groupoid representations},
  journal = {Demonstratio Mathematica},
  volume  = {44},
  number  = {1},
  pages   = {29--48},
  year    = {2011},
  doi     = {10.1515/dema-2013-0294}
}
\end{document}